\documentclass[a4paper,oneside,10pt]{article}%
\usepackage{amsmath}
\usepackage{amsfonts}
\usepackage{amssymb}
\usepackage{graphicx}
\usepackage[square,numbers,sort&compress]{natbib}%
\providecommand{\U}[1]{\protect \rule{.1in}{.1in}}

\newtheorem{theorem}{Theorem}[section]

\newtheorem{definition}[theorem]{Definition}

\newtheorem{lemma}[theorem]{Lemma}

\newtheorem{proposition}[theorem]{Proposition}
\newtheorem{remark}[theorem]{Remark}

\newenvironment{proof}[1][Proof]{\noindent \textbf{#1.} }{\  $\Box$}
\numberwithin{equation}{section}

\begin{document}

\title{L\'{e}vy's martingale characterization and reflection principle of
$G$-Brownian motion}
\author{Mingshang Hu \thanks{Zhongtai Securities Institute for Financial Studies,
Shandong University, Jinan, Shandong 250100, China. humingshang@sdu.edu.cn.
Research supported by NSF (No. 11671231) and Young Scholars Program of
Shandong University (No. 2016WLJH10)}
\and Xiaojun Ji \thanks{School of Mathematics, Shandong University, Jinan, Shandong
250100, PR China. jixiaojun@mail.sdu.edu.cn}
\and Guomin Liu \thanks{Zhongtai Securities Institute for Financial Studies, Shandong University, Jinan,
Shandong 250100, PR China. sduliuguomin@163.com. }}
\maketitle

\textbf{Abstract}. In this paper, we obtain L\'{e}vy's martingale
characterization of $G$-Brownian motion without the nondegenerate condition.
Base on this characterization, we prove the reflection principle of
$G$-Brownian motion. Furthermore, we use Krylov's estimate to get the
reflection principle of $\tilde{G}$-Brownian motion.



\addcontentsline{toc}{section}{\hspace*{1.8em}Abstract}

\section{Introduction}

Motivated by the model uncertainty in financial market, Peng in
\cite{peng2005, peng2008} first introduced the notions of \thinspace
$G$-expectation, conditional $G$-expectation and $G$-Brownian motion via the
following $G$-heat equation:%
\begin{equation}
\partial_{t}u-G(D_{x}^{2}u)=0,\text{ }u(0,x)=\varphi(x), \label{intro-1}%
\end{equation}
where $G:\mathbb{S}(d)\rightarrow \mathbb{R}$ is a monotonic and sublinear
function. In \cite{Peng 1}, Peng further extended the definition of
$G$-expectation and $G$-Brownian motion to $\tilde{G}$-expectation and
$\tilde{G}$-Brownian motion via the following PDE:%
\begin{equation}
\partial_{t}u-\tilde{G}(D_{x}^{2}u)=0,\text{ }u(0,x)=\varphi(x),
\label{intro-2}%
\end{equation}
where $\tilde{G}:\mathbb{S}(d)\rightarrow \mathbb{R}$ is monotonic and
dominated by $G$. Based on $G$-Brownian motion, Peng established the related
It\^{o}'s stochastic calculus theory, for example, It\^{o}'s integral,
It\^{o}'s formula, $G$-martingale and $G$-stochastic differential equation
($G$-SDE).

It is well-known that the reflection principle for classical Brownian motion
is an important result. In this paper, we study the reflection principle for
$G$-Brownian motion or, more generally, $\tilde{G}$-Brownian motion, i.e.,
calculate the distribution of the following process:%
\[
\sup_{s\leq t}(B_{s}-B_{t}),\text{ }t\geq0,
\]
where $B$ is a $G$-Brownian motion or $\tilde{G}$-Brownian motion. In order to
obtain the reflection principle, we first investigate the L\'{e}vy's martingale
characterization of $G$-Brownian motion.

In \cite{X-Z, X-Z-1}, Xu et al. obtained the L\'{e}vy's martingale
characterization theorem for $1$-dimensional symmetric $G$-martingale under
the nondegenerate condition. Song \cite{Song} also proved the
martingale characterization theorem for $1$-dimensional symmetric
$G$-martingale under the nondegenerate condition by a different method. In
this paper, we consider general symmetric martingales, which may not be
$G$-martingale, without the nondegenerate condition. We introduce a new kind
of discrete product space method for martingale (see Lemmas \ref{newlem1} and
\ref{newlem2}) to deal with the degenerate case, and obtain the corresponding
L\'{e}vy's martingale characterization of $G$-Brownian motion. In the symmetric
$G$-martingale case, we give a simpler and more direct method. Based on the martingale characterization of $G$-Brownian motion, we obtain the
reflection principle for $G$-Brownian motion.

However, the martingale characterization method does not hold for $\tilde{G}%
$-Brownian motion, because we do not have the L\'{e}vy martingale
characterization of $\tilde{G}$-Brownian motion. In this case, we first use
the estimate in \cite{H-W-Z} to give a discrete approximation of process
$($sgn$(B_{t}))_{t\leq T}$ (see Lemma \ref{sgnB}) under the nondegenerate
condition. Then we obtain the related reflection principle under the
nondegenerate condition. Finally, we apply the Krylov's estimate to get the
related reflection principle without the nondegenerate condition.

The paper is organized as follows. In section 2, we recall some basic notions
and results of sublinear expectation and $G$-Brownian motion. In section 3, we
give the definition of consistent sublinear expectation space, and extend
Peng's definition of stochastic calculus for $G$-Brownian motion to a kind of
martingales on the consistent sublinear expectation space. The L\'{e}vy's
martingale characterization of $G$-Brownian motion is stated and proved in
section 4. In section 5, we obtain the reflection principle for $G$-Brownian
motion and $\tilde{G}$-Brownian motion.

\section{Preliminaries}

In this section, we recall some basic notions and results of sublinear
expectation and $G$-Brownian motion. The readers may refer to \cite{peng2005,
peng2008, Peng 1, D-H-P, Gao, H-P} for more details.

Let $\Omega$ be a given set and $\mathcal{H}$ be a linear space of real-valued
functions on $\Omega$ such that $c\in \mathcal{H}$ for all constants $c$ and
$|X|\in \mathcal{H}$ if $X\in \mathcal{H}$. We further suppose that if
$X_{1},\ldots,X_{n}\in \mathcal{H}$, then $\varphi(X_{1},\ldots,X_{n}%
)\in \mathcal{H}$ for each $\varphi \in C_{b.Lip}(\mathbb{R}^{n})$, where
$C_{b.Lip}(\mathbb{R}^{n})$ denotes the set of all bounded and Lipschitz
functions on $\mathbb{R}^{n}$. $\mathcal{H}$ is considered as the space of
random variables. $X=(X_{1},X_{2},\dots,X_{n})^{T}$, $X_{i}\in \mathcal{H}$, is
called an $n$-dimensional random vector, denoted by $X\in \mathcal{H}^{n}$,
where $^{T}$ denotes the transpose.

\begin{definition}
\label{sublinear expectation} A sublinear expectation is a functional
$\hat{\mathbb{E}}:\mathcal{H}\rightarrow \mathbb{R}$ satisfying the following
properties: for each $X,Y\in \mathcal{H}$, we have

\begin{description}
\item[(1)] \textbf{Monotonicity:} $X\geq Y$ implies $\hat{\mathbb{E}}%
[X]\geq \hat{\mathbb{E}}[Y]$;

\item[(2)] \textbf{Constant preserving:} $\hat{\mathbb{E}}[c]=c$ for
$c\in \mathbb{R}$;

\item[(3)] \textbf{Sub-additivity:} $\hat{\mathbb{E}}[X+Y]\leq \hat{\mathbb{E}%
}[X]+\hat{\mathbb{E}}[Y]$;

\item[(4)] \textbf{Positive homogeneity:} $\hat{\mathbb{E}}[\lambda
X]=\lambda \hat{\mathbb{E}}[X]$ for $\lambda>0$.
\end{description}

The triple $(\Omega,\mathcal{H},\hat{\mathbb{E}})$ is called a sublinear
expectation space. If $(1)$ and $(2)$ are satisfied, $\hat{\mathbb{E}}$ is
called a nonlinear expectation and the triple $(\Omega,\mathcal{H}%
,\hat{\mathbb{E}})$ is called a nonlinear expectation space.
\end{definition}

\begin{remark}
If the inequality in (3) becomes equality, $\hat{\mathbb{E}}$ is called a
linear expectation and $(\Omega,\mathcal{H},\hat{\mathbb{E}})$ is called a
linear expectation space.
\end{remark}

For each given $p\geq1$, in order to define the $p$-norm on $\mathcal{H}$,
suppose $|X|^{p}\in \mathcal{H}$ for each $X\in \mathcal{H}$. We denote by
$\mathbb{L}^{p}(\Omega)$ the completion of $\mathcal{H}$ under the norm
$||X||_{\mathbb{L}^{p}}:=(\hat{\mathbb{E}}[|X|^{p}])^{1/p}$. It is easy to
verify that $\mathbb{L}^{p}(\Omega)\subset \mathbb{L}^{p^{\prime}}(\Omega)$ for
each $1\leq p^{\prime}\leq p$. Note that $|\hat{\mathbb{E}}[X]-\hat
{\mathbb{E}}[Y]|\leq \hat{\mathbb{E}}[|X-Y|]$, then $\hat{\mathbb{E}}$ can be
continuously extended to the mapping from $\mathbb{L}^{1}(\Omega)$ to
$\mathbb{R}$. One can check that $(\Omega,\mathbb{L}^{1}(\Omega),\hat
{\mathbb{E}})$\ forms a sublinear expectation space, which is called a
complete sublinear expectation space.

We say that a sequence $\{X_{n}\}_{n=1}^{\infty}$ converges to $X$ in
$\mathbb{L}^{p}$, denoted by $X=\mathbb{L}^{p}-\lim_{n\rightarrow \infty}X_{n}%
$, if
\[
\lim \limits_{n \rightarrow \infty}\hat{\mathbb{E}}[|X_{n}-X|^{p}]=0.
\]

Now we give the notions of identical distribution and independence.

\begin{definition}
Let $X_{1}$ and $X_{2}$ be two $n$-dimensional random vectors on complete
sublinear expectation spaces $(\Omega_{1},\mathbb{L}^{1}(\Omega_{1}%
),\hat{\mathbb{E}}_{1})$ and $(\Omega_{2},\mathbb{L}^{1}(\Omega_{2}%
),\hat{\mathbb{E}}_{2})$, respectively. They are called identically
distributed, denoted by $X_{1}\overset{d}{=}X_{2}$, if for each $\varphi \in
C_{b.Lip}(\mathbb{R}^{n})$,
\[
\hat{\mathbb{E}}_{1}[\varphi(X_{1})]=\hat{\mathbb{E}}_{2}[\varphi(X_{2})].
\]

\end{definition}

\begin{definition}
Let $(\Omega,\mathbb{L}^{1}(\Omega),\hat{\mathbb{E}})$ be a complete sublinear
expectation space. An $m$-dimensional random vector $Y$ is said to be
independent from an $n$-dimensional random vector $X$ if for each $\varphi \in
C_{b.Lip}(\mathbb{R}^{n+m})$,
\[
\hat{\mathbb{E}}[\varphi(X,Y)]=\hat{\mathbb{E}}[\hat{\mathbb{E}}%
[\varphi(x,Y)]_{x=X}].
\]

\end{definition}

In the following, we review $G$-normal distribution and $G$-Brownian motion.

Let $G:\mathbb{S}(d)\rightarrow \mathbb{R}$ be a given monotonic and sublinear
function, i.e., for $A,A^{\prime}\in \mathbb{S}(d)$,
\[%
\begin{cases}
G\left(  A\right)  \geq G\left(  A^{\prime}\right)  \text{ if }A\geq
A^{\prime},\\
G\left(  A+A^{\prime}\right)  \leq G(A)+G(A^{\prime}),\\
G\left(  \lambda A\right)  =\lambda G\left(  A\right)  \text{ for }\lambda
\geq0,
\end{cases}
\]
where $\mathbb{S}(d)$ denotes the collection of all $d\times d$ symmetric
matrices. By the Hahn-Banach theorem, one can check that there exists a
bounded, convex and closed subset $\Gamma \subset \mathbb{S}_{+}(d)$ such that%
\begin{equation}
G(A)=\frac{1}{2}\sup_{\gamma \in \Gamma}\text{\textrm{tr}}[\gamma A]\text{ for
}A\in \mathbb{S}(d), \label{neweq-1}%
\end{equation}
where $\mathbb{S}_{+}(d)$ denotes the collection of nonnegative elements in
$\mathbb{S}(d)$. From (\ref{neweq-1}), it is easy to deduce that there exists
a constant $C>0$ such that for each $A,A^{\prime}\in\mathbb{S}(d)$,
\[
|G(A)-G(A^{\prime})|\leq C|A-A^{\prime}|.
\]

\begin{definition}
Let $(\Omega,\mathbb{L}^{1}(\Omega),\hat{\mathbb{E}})$ be a complete sublinear
expectation space. A $d$-dimensional random vector $X$ is called $G$-normally
distributed, denoted by $X\overset{d}{=}N(0,\Gamma)$, if for each $\varphi \in
C_{b.Lip}(\mathbb{R}^{d}),$
\[
\hat{\mathbb{E}}[\varphi(X)]=u^{\varphi}(1,0),
\]
where $u^{\varphi}$ is the viscosity solution of the following $G$-heat
equation
\begin{equation}%
\begin{cases}
\partial_{t}u-G(D_{x}^{2}u)=0,\\
u(0,x)=\varphi(x).
\end{cases}
\label{GPDE}%
\end{equation}

\end{definition}

\begin{remark}
If $G(A)=\frac{1}{2}$\textrm{tr}$[A]\ $for$\ A\in \mathbb{S}(d)$, then
$\Gamma=\{I_{d}\}$ and $G$-normal distribution is the classical standard normal distribution.
\end{remark}

\begin{definition}
Let $(\Omega,\mathbb{L}^{1}(\Omega),\hat{\mathbb{E}})$ be a complete sublinear
expectation space. A $d$-dimensional process $(B_{t})_{t\geq0}$ with $B_{t}%
\in(\mathbb{L}^{1}(\Omega))^{d}$ for each $t\geq0$ is called a $G$-Brownian
motion if the following properties are satisfied:

\begin{description}
\item[(i)] $B_{0}=0$;

\item[(ii)] For each $t,s\geq0$, $B_{t+s}-B_{t}\overset{d}{=}N(0,s\Gamma)$ and
is independent from $(B_{t_{1}},\ldots,B_{t_{n}})$, for each $n\in \mathbb{N}$
and $0\leq t_{1}\leq \cdots \leq t_{n}\leq t$.
\end{description}
\end{definition}

\begin{remark}
\label{newre-000} It is easy to verify that for each $\varphi \in
C_{b.Lip}(\mathbb{R}^{d})$, $\hat{\mathbb{E}}[\varphi(B_{t+s}-B_{t}%
)]=u^{\varphi}(s,0)$, where $u^{\varphi}$ is the solution of the $G$-heat
equation (\ref{GPDE}). If $G(A)=\frac{1}{2}$\textrm{tr}$[A]\ $for$\ A\in
\mathbb{S}(d)$, then $G$-Brownian motion is classical standard Brownian motion.
\end{remark}

For each $t\geq0$, set
\[
L_{ip}(\Omega_{t})=\{ \varphi(B_{t_{1}},B_{t_{2}}-B_{t_{1}},\cdots,B_{t_{n}%
}-B_{t_{n-1}}):n\in \mathbb{N},0\leq t_{1}<\cdots<t_{n}\leq t,\varphi \in
C_{b.Lip}(\mathbb{R}^{n\times d})\}
\]
and
\[
L_{ip}(\Omega):=\bigcup_{m=1}^{\infty}L_{ip}(\Omega_{m}).
\]
For each $p\geq1$, we denote by $L_{G}^{p}(\Omega_{t})$ (resp.$\ L_{G}%
^{p}(\Omega)$) the completion of $L_{ip}(\Omega_{t})$ (resp.$\ L_{ip}(\Omega
)$) under the norm $||X||_{L_{G}^{p}}:=(\hat{\mathbb{E}}[|X|^{p}])^{\frac
{1}{p}}$ for each $t\geq0$. For each $\xi=\varphi(B_{t_{1}},B_{t_{2}}%
-B_{t_{1}},\cdots,B_{t_{n}}-B_{t_{n-1}})\in L_{ip}(\Omega)$, define the
conditional expectation of $\xi$ at $t_{i}$, $i\leq n$, by
\begin{equation}
\hat{\mathbb{E}}_{t_{i}}[\varphi(B_{t_{1}},B_{t_{2}}-B_{t_{1}},\cdots
,B_{t_{n}}-B_{t_{n-1}})]=\psi(B_{t_{1}},\cdots,B_{t_{i}}-B_{t_{i-1}}),
\label{new-newww-2}%
\end{equation}
where $\psi(x_{1},\cdots,x_{i})=\hat{\mathbb{E}}[\varphi(x_{1},\cdots
,x_{i},B_{t_{i+1}}-B_{t_{i}},\cdots,B_{t_{n}}-B_{t_{n-1}})]$. It is easy to
check that $\hat{\mathbb{E}}_{t_{i}}$ still satisfies (1)-(4) in Definition
\ref{sublinear expectation} and can be continuously extended to the mapping
from $L_{G}^{1}(\Omega)$ to $L_{G}^{1}(\Omega_{t_{i}})$. $(\Omega,L_{G}%
^{1}(\Omega),(L_{G}^{1}(\Omega_{t}))_{t\geq0},(\hat{\mathbb{E}}_{t})_{t\geq
0})$ is called the $G$-expectation space.

\section{Consistent sublinear expectation space}

In this section, we first give the definition of consistent sublinear
expectation space, and then we extend Peng's definition of stochastic calculus
for $G$-Brownian motion to a kind of martingales on the consistent sublinear
expectation space. We only consider the continuous time case, and the
definitions and results still hold for discrete time case.

Let $\Omega$ be a given set and $(\mathcal{H}_{t})_{t\geq0}$, $\mathcal{H}$ be
a family of linear space of real-valued functions on $\Omega$ such that

\begin{description}
\item[(i)] $\mathcal{H}_{0}=\mathbb{R}$ and $\mathcal{H}_{s}\subset
\mathcal{H}_{t}$ for each $0\leq s\leq t$;

\item[(ii)] If $X\in \mathcal{H}_{t}$ (resp. $\mathcal{H}$), then
$|X|\in \mathcal{H}_{t}$ (resp. $\mathcal{H}$);

\item[(iii)] If $X_{1},\ldots,X_{n}\in \mathcal{H}_{t}$ (resp. $\mathcal{H}$),
then $\varphi(X_{1},\ldots,X_{n})\in \mathcal{H}_{t}$ (resp. $\mathcal{H}$) for
each $\varphi \in C_{b.Lip}(\mathbb{R}^{n})$.
\end{description}

\begin{definition}
\label{sublinear conditional expectation} A consistent sublinear expectation
on $(\mathcal{H}_{t})_{t\geq0}$ is a family of mappings $\hat{\mathbb{E}}%
_{t}:\mathcal{H}\rightarrow \mathcal{H}_{t}$ which satisfies the following
properties: for each $X,Y\in \mathcal{H}$,

\begin{description}
\item[(1)] \textbf{Monotonicity:} $X\geq Y$ implies $\hat{\mathbb{E}}%
_{t}[X]\geq \hat{\mathbb{E}}_{t}[Y]$;

\item[(2)] \textbf{Constant preserving:} $\hat{\mathbb{E}}_{t}[\eta]=\eta$ for
$\eta \in \mathcal{H}_{t}$;

\item[(3)] \textbf{Sub-additivity:} $\hat{\mathbb{E}}_{t}[X+Y]\leq
\hat{\mathbb{E}}_{t}[X]+\hat{\mathbb{E}}_{t}[Y]$;

\item[(4)] \textbf{Positive homogeneity:} $\hat{\mathbb{E}}_{t}[\eta
X]=\eta \hat{\mathbb{E}}_{t}[X]$ for each positive and bounded $\eta
\in \mathcal{H}_{t}$;

\item[(5)] \textbf{Consistency:} $\hat{\mathbb{E}}_{s}[\hat{\mathbb{E}}%
_{t}[X]]=\hat{\mathbb{E}}_{s}[X]$ for $s\leq t$.
\end{description}

The triple $(\Omega,\mathcal{H},(\mathcal{H}_{t})_{t\geq0},(\hat{\mathbb{E}%
}_{t})_{t\geq0})$ is called a consistent sublinear expectation space and
$\hat{\mathbb{E}}_{t}$ is called a sublinear conditional expectation.
\end{definition}

\begin{remark}
We always denote $\hat{\mathbb{E}}:=\hat{\mathbb{E}}_{0}$ in the above
definition. Obviously, $(\Omega,\mathcal{H},\hat{\mathbb{E}})$ forms a
sublinear expectation space.
\end{remark}

\begin{remark}
One can check that the $G$-expectation space $(\Omega,L_{G}^{1}(\Omega
),(L_{G}^{1}(\Omega_{t}))_{t\geq0},(\hat{\mathbb{E}}_{t})_{t\geq0})$ satisfies
(1)-(5) in Definition \ref{sublinear conditional expectation}.
\end{remark}

For each given $p\geq1$, in order to define the $p$-norm on $\mathcal{H}_{t}$
(resp. $\mathcal{H}$), suppose $|X|^{p}\in \mathcal{H}_{t}$ (resp.
$\mathcal{H}$) for each $X\in \mathcal{H}_{t}$ (resp. $\mathcal{H}$), and we
denote by $\mathbb{L}^{p}(\Omega_{t})$ (resp. $\mathbb{L}^{p}(\Omega)$) the
completion of $\mathcal{H}_{t}$ (resp. $\mathcal{H}$) under the norm
$||X||_{\mathbb{L}^{p}}:=(\hat{\mathbb{E}}[|X|^{p}])^{1/p}$. Note that
$\hat{\mathbb{E}}[|\hat{\mathbb{E}}_{t}[X]-\hat{\mathbb{E}}_{t}[Y]|]\leq
\hat{\mathbb{E}}[|X-Y|]$, then $\hat{\mathbb{E}}_{t}$ can be continuously
extended to the mapping from $\mathbb{L}^{1}(\Omega)$ to $\mathbb{L}%
^{1}(\Omega_{t})$. It is easy to check that $(\Omega,\mathbb{L}^{1}%
(\Omega),(\mathbb{L}^{1}(\Omega_{t}))_{t\geq0},(\hat{\mathbb{E}}_{t})_{t\geq
0})$\ forms a consistent sublinear expectation space, which is called a
complete consistent sublinear expectation space.

The following properties are very useful in sublinear analysis.

\begin{proposition}
\label{newpro01} Let $(\Omega,\mathbb{L}^{1}(\Omega),(\mathbb{L}^{1}%
(\Omega_{t}))_{t\geq0},(\hat{\mathbb{E}}_{t})_{t\geq0})$ be a complete
consistent sublinear expectation space. Assume $X\in(\mathbb{L}^{1}(\Omega
_{t}))^{n}$ and $Y\in(\mathbb{L}^{1}(\Omega))^{m}$. Then%
\begin{equation}
\hat{\mathbb{E}}_{t}[\varphi(X,Y)]=\hat{\mathbb{E}}_{t}[\varphi(x,Y)]_{x=X}%
\text{ for each }\varphi \in C_{b.Lip}(\mathbb{R}^{n+m}). \label{new-newww-1}%
\end{equation}
Specially, for each bounded $\eta \in \mathbb{L}^{1}(\Omega_{t})$,%
\[
\hat{\mathbb{E}}_{t}[\eta Z]=\eta^{+}\hat{\mathbb{E}}_{t}[Z]+\eta^{-}%
\hat{\mathbb{E}}_{t}[-Z]\text{ for each }Z\in \mathbb{L}^{1}(\Omega).
\]

\end{proposition}

\begin{proof}
For each fixed $N>0$, denote $B_{N}(0):=\{x\in \mathbb{R}^{n}:|x|\leq N\}$. For
each integer $n\geq1$, by partition of unity theorem, there exist $h_{i}%
^{n}\in C_{0}^{\infty}(\mathbb{R}^{n})$, $i=1,\ldots,k_{n}$, such that the
diameter of support $\lambda($supp$(h_{i}^{n}))\leq1/n$, $0\leq h_{i}^{n}%
\leq1$ and $I_{B_{N}(0)}(x)\leq \sum_{i=1}^{k_{n}}h_{i}^{n}(x)\leq1$. Choose
$x_{i}^{n}\in \mathbb{R}^{n}$ such that $h_{i}^{n}(x_{i}^{n})>0$ and $L>0$ such
that $|\varphi|\leq L$. Note that%
\begin{align*}
&  \left \vert \sum_{i=1}^{k_{n}}h_{i}^{n}(X)\hat{\mathbb{E}}_{t}[\varphi
(x_{i}^{n},Y)]-\hat{\mathbb{E}}_{t}[\varphi(X,Y)]\right \vert \\
&  \leq \sum_{i=1}^{k_{n}}h_{i}^{n}(X)\left \vert \hat{\mathbb{E}}_{t}%
[\varphi(x_{i}^{n},Y)]-\hat{\mathbb{E}}_{t}[\varphi(X,Y)]\right \vert +\left(
1-\sum_{i=1}^{k_{n}}h_{i}^{n}(X)\right)  \left \vert \hat{\mathbb{E}}%
_{t}[\varphi(X,Y)]\right \vert \\
&  \leq \sum_{i=1}^{k_{n}}h_{i}^{n}(X)\left \vert \hat{\mathbb{E}}_{t}%
[\varphi(x_{i}^{n},Y)]-\hat{\mathbb{E}}_{t}[\varphi(X,Y)]\right \vert
+\frac{L|X|}{N}\\
&  \leq C_{\varphi}\sum_{i=1}^{k_{n}}h_{i}^{n}(X)|X-x_{i}^{n}|+\frac{L|X|}%
{N}\\
&  \leq \frac{C_{\varphi}}{n}+\frac{L|X|}{N}%
\end{align*}
and%
\begin{align*}
&  \left \vert \sum_{i=1}^{k_{n}}h_{i}^{n}(X)\hat{\mathbb{E}}_{t}[\varphi
(x_{i}^{n},Y)]-\hat{\mathbb{E}}_{t}[\varphi(x,Y)]_{x=X}\right \vert \\
&  \leq \sum_{i=1}^{k_{n}}h_{i}^{n}(X)\left \vert \hat{\mathbb{E}}_{t}%
[\varphi(x_{i}^{n},Y)]-\hat{\mathbb{E}}_{t}[\varphi(x,Y)]_{x=X}\right \vert
+\frac{L|X|}{N}\\
&  \leq C_{\varphi}\sum_{i=1}^{k_{n}}h_{i}^{n}(X)|X-x_{i}^{n}|+\frac{L|X|}%
{N}\\
&  \leq \frac{C_{\varphi}}{n}+\frac{L|X|}{N},
\end{align*}
where $C_{\varphi}$ is the Lipschitz constant of $\varphi$, then we obtain%
\[
\left \vert \hat{\mathbb{E}}_{t}[\varphi(X,Y)]-\hat{\mathbb{E}}_{t}%
[\varphi(x,Y)]_{x=X}\right \vert \leq \frac{2C_{\varphi}}{n}+\frac{2L|X|}{N}.
\]
Letting $n\rightarrow \infty$ first and then $N\rightarrow \infty$, we get
$\hat{\mathbb{E}}_{t}[\varphi(X,Y)]=\hat{\mathbb{E}}_{t}[\varphi(x,Y)]_{x=X}$.
By positive homogeneity of $\hat{\mathbb{E}}_{t}$, we can deduce
$\hat{\mathbb{E}}_{t}[\eta Z]=\eta^{+}\hat{\mathbb{E}}_{t}[Z]+\eta^{-}%
\hat{\mathbb{E}}_{t}[-Z]$.
\end{proof}

\begin{proposition}
\label{mean certainty} Let $(\Omega,\mathbb{L}^{1}(\Omega),(\mathbb{L}%
^{1}(\Omega_{t}))_{t\geq0},(\hat{\mathbb{E}}_{t})_{t\geq0})$ be a complete
consistent sublinear expectation space. Assume $Y\in \mathbb{L}^{1}(\Omega)$
such that $\hat{\mathbb{E}}_{t}[Y]=-\hat{\mathbb{E}}_{t}[-Y]$ for some
$t\geq0$. Then, for each $X\in \mathbb{L}^{1}(\Omega)$ and bounded $\eta
\in \mathbb{L}^{1}(\Omega_{t})$, we have
\[
\hat{\mathbb{E}}_{t}[X+\eta Y]=\hat{\mathbb{E}}_{t}[X]+\eta \hat{\mathbb{E}%
}_{t}[Y].
\]

\end{proposition}

\begin{proof}
By Proposition \ref{newpro01} and $\hat{\mathbb{E}}_{t}[Y]=-\hat{\mathbb{E}%
}_{t}[-Y]$, we get
\[
\hat{\mathbb{E}}_{t}[\eta Y]=\eta^{+}\hat{\mathbb{E}}_{t}[Y]+\eta^{-}%
\hat{\mathbb{E}}_{t}[-Y]=\eta^{+}\hat{\mathbb{E}}_{t}[Y]-\eta^{-}%
\hat{\mathbb{E}}_{t}[Y]=\eta \hat{\mathbb{E}}_{t}[Y].
\]
Similarly, we have $\hat{\mathbb{E}}_{t}[-\eta Y]=-\eta \hat{\mathbb{E}}%
_{t}[Y]$. Then by the sub-additivity of $\hat{\mathbb{E}}_{t}$, we obtain
\[
\hat{\mathbb{E}}_{t}[X+\eta Y]\leq \hat{\mathbb{E}}_{t}[X]+\hat{\mathbb{E}}%
_{t}[\eta Y]=\hat{\mathbb{E}}_{t}[X]+\eta \hat{\mathbb{E}}_{t}[Y]
\]
and
\[
\hat{\mathbb{E}}_{t}[X]+\eta \hat{\mathbb{E}}_{t}[Y]=\hat{\mathbb{E}}%
_{t}[X]-\hat{\mathbb{E}}_{t}[-\eta Y]\leq \hat{\mathbb{E}}_{t}[X+\eta Y].
\]
Thus
\[
\hat{\mathbb{E}}_{t}[X+\eta Y]=\hat{\mathbb{E}}_{t}[X]+\eta \hat{\mathbb{E}%
}_{t}[Y].
\]

\end{proof}

Now we give the definitions of adapted processes and martingales on the
consistent sublinear expectation space.

\begin{definition}
Let $(\Omega,\mathbb{L}^{1}(\Omega),(\mathbb{L}^{1}(\Omega_{t}))_{t\geq
0},(\hat{\mathbb{E}}_{t})_{t\geq0})$ be a complete consistent sublinear
expectation space. A $d$-dimensional adapted process is a family of random
vectors $(X_{t})_{t\geq0}$ such that $X_{t}\in(\mathbb{L}^{1}(\Omega_{t}%
))^{d}$ for each $t\geq0$. Two $d$-dimensional adapted processes $X^{i}%
=(X_{t}^{i})_{t\geq0}$, $i=1,2$, on complete consistent sublinear expectation
spaces $(\Omega^{i},\mathbb{L}^{1}(\Omega^{i}),(\mathbb{L}^{1}(\Omega_{t}%
^{i}))_{t\geq0},(\hat{\mathbb{E}}_{t}^{i})_{t\geq0})$ are called identically
distributed, denoted by $X^{1}\overset{d}{=}X^{2}$, if for each $n\in
\mathbb{N}$, $0\leq t_{1}<\cdots<t_{n}<\infty$, $\varphi \in C_{b.Lip}%
(\mathbb{R}^{n\times d})$,
\[
\hat{\mathbb{E}}^{1}[\varphi(X_{t_{1}}^{1},X_{t_{2}}^{1}-X_{t_{1}}^{1}%
,\cdots,X_{t_{n}}^{1}-X_{t_{n-1}}^{1})]=\mathbb{\hat{E}}^{2}[\varphi(X_{t_{1}%
}^{2},X_{t_{2}}^{2}-X_{t_{1}}^{2},\cdots,X_{t_{n}}^{2}-X_{t_{n-1}}^{2})].
\]

\end{definition}

\begin{definition}
Let $(\Omega,\mathbb{L}^{1}(\Omega),(\mathbb{L}^{1}(\Omega_{t}))_{t\geq
0},(\hat{\mathbb{E}}_{t})_{t\geq0})$ be a complete consistent sublinear
expectation space. A $d$-dimensional adapted process $M_{t}=(M_{t}^{1}%
,\ldots,M_{t}^{d})^{T}$, $t\geq0$, is called a martingale if $\hat{\mathbb{E}%
}_{s}[M_{t}^{i}]=M_{s}^{i}$ for each $s\leq t$ and $i\leq d$. Furthermore, a
$d$-dimensional martingale $M$ is called symmetric if $\hat{\mathbb{E}}%
[M_{t}^{i}]=-\hat{\mathbb{E}}[-M_{t}^{i}]$ for each $t\geq0$ and $i\leq d$.
\end{definition}

\begin{remark}
In the above definition, $\hat{\mathbb{E}}[M_{t}^{i}]=-\hat{\mathbb{E}}%
[-M_{t}^{i}]$ implies $\hat{\mathbb{E}}_{s}[M_{t}^{i}]=-\hat{\mathbb{E}}%
_{s}[-M_{t}^{i}]$ for each $s\in \lbrack0,t]$, which is deduced by
$\hat{\mathbb{E}}_{s}[M_{t}^{i}]+\hat{\mathbb{E}}_{s}[-M_{t}^{i}]\geq0$ and
\[
\hat{\mathbb{E}}[\hat{\mathbb{E}}_{s}[M_{t}^{i}]+\hat{\mathbb{E}}_{s}%
[-M_{t}^{i}]]\leq \hat{\mathbb{E}}[\hat{\mathbb{E}}_{s}[M_{t}^{i}%
]]+\hat{\mathbb{E}}[\hat{\mathbb{E}}_{s}[-M_{t}^{i}]]=\hat{\mathbb{E}}%
[M_{t}^{i}]+\hat{\mathbb{E}}[-M_{t}^{i}]=0.
\]

\end{remark}

In the following, we construct the stochastic calculus with respect to a kind
of martingales. For simplicity, we first discuss the $1$-dimensional case.

Let $M$ be a symmetric martingale such that $M_{0}=0$, $M_{t}\in \mathbb{L}%
^{2}(\Omega_{t})$ for each $t\geq0$ and%

\begin{equation}
\hat{\mathbb{E}}_{t}[|M_{t+s}-M_{t}|^{2}]\leq Cs,\text{ for each }t,s\geq0,
\label{M}%
\end{equation}
where $C>0$ is a constant.

Let $T>0$ be given. Set%
\[
\mathbb{M}^{2,0}(0,T):=\{ \eta_{t}=\sum_{i=0}^{n-1}\xi_{i}I_{[t_{i},t_{i+1}%
)}(t):0=t_{0}<\cdots<t_{n}=T,\xi_{i}\in \mathbb{L}^{2}(\Omega_{t_{i}})\}.
\]
We denote by $\mathbb{M}^{2}(0,T)$ the completion of $\mathbb{M}^{2,0}(0,T)$
under the norm $||\eta||_{\mathbb{M}^{2}}=(\hat{\mathbb{E}}[\int_{0}^{T}%
|\eta_{t}|^{2}dt])^{1/2}$.

For each $\eta \in \mathbb{M}^{2,0}(0,T)$ with $\eta_{t}=\sum_{i=0}^{n-1}\xi
_{i}I_{[t_{i},t_{i+1})}(t)$, define the stochastic integral with respect to
$M$ as,
\[
I(\eta)=\int_{0}^{T}\eta_{t}dM_{t}:=\sum_{i=0}^{n-1}\xi_{{i}}(M_{t_{i+1}%
}-M_{t_{i}}).
\]

The proof of the following Lemma is the same to the proof of Lemma 3.5 in
\cite{Peng 1}. We omit it here.

\begin{lemma}
\label{Mcontrol} The mapping $I:\mathbb{M}^{2,0}(0,T)\rightarrow \mathbb{L}%
^{2}(\Omega_{T})$ is a continuous linear mapping and can be continuously
extended to $I:\mathbb{M}^{2}(0,T)\rightarrow \mathbb{L}^{2}(\Omega_{T})$. For
each $\eta \in \mathbb{M}^{2}(0,T)$, $(\int_{0}^{t}\eta_{s}dM_{s})_{t\leq T}$ is
a symmetric martingale and
\begin{equation}
\hat{\mathbb{E}}[|\int_{0}^{T}\eta_{t}dM_{t}|^{2}]\leq C\hat{\mathbb{E}}%
[\int_{0}^{T}|\eta_{t}|^{2}dt]. \label{M continuity}%
\end{equation}

\end{lemma}

For multi-dimensional case, let $M=(M^{1},\cdots,M^{d})^{T}$ be a
$d$-dimensional symmetric martingale with $M^{i}$ satisfying (\ref{M}) for
$i=1,\ldots,d$. Then for each $(n\times d)$-dimensional process $\eta
=(\eta^{ij})$ with $\eta^{ij}\in \mathbb{M}^{2}(0,T)$, the definition of the
stochastic integral $\int_{0}^{T}\eta_{t}dM_{t}$ can be defined through
$1$-dimensional integral just like the classical case.

Now we consider the quadratic variation process of the $1$-dimensional
symmetric martingale $M$ satisfying (\ref{M}). For any $t\geq0$, let $\pi
_{t}^{n}=\{t_{0}^{n},...,t_{n}^{n}\}$ be a partition of $[0,t]$ and denote
\[
\mu(\pi_{t}^{n}):=\max \{|t_{i+1}^{n}-t_{i}^{n}|:i=0,1,\cdots,n-1\}.
\]
Note that
\begin{align*}
\sum_{i=0}^{n-1}(M_{t_{i+1}^{n}}-M_{t_{i}^{n}})^{2}  &  =\sum_{i=0}%
^{n-1}(M_{t_{i+1}^{n}}^{2}-M_{t_{i}^{n}}^{2})-2\sum_{i=0}^{n-1}M_{t_{i}^{n}%
}(M_{t_{i+1}^{n}}-M_{t_{i}^{n}})\\
&  =M_{t}^{2}-2\sum_{i=0}^{n-1}M_{t_{i}^{n}}(M_{t_{i+1}^{n}}-M_{t_{i}^{n}}),
\end{align*}
and
\begin{align*}
\hat{\mathbb{E}}[\int_{0}^{t}\sum_{i=0}^{n-1}|M_{t_{i}^{n}}I_{[t_{i}%
^{n},t_{i+1}^{n})}(s)-M_{s}|^{2}ds]  &  =\hat{\mathbb{E}}[\sum_{i=0}^{n-1}%
\int_{t_{i}^{n}}^{t_{i+1}^{n}}|M_{t_{i}^{n}}-M_{s}|^{2}ds]\\
&  \leq \sum_{i=0}^{n-1}\int_{t_{i}^{n}}^{t_{i+1}^{n}}\hat{\mathbb{E}%
}[|M_{t_{i}^{n}}-M_{s}|^{2}]ds\\
&  \leq \frac{Ct}{2}\mu(\pi_{t}^{n}),
\end{align*}
then $\sum_{i=0}^{n-1}(M_{t_{i+1}^{n}}-M_{t_{i}^{n}})^{2}$ converges to
$M_{t}^{2}-2\int_{0}^{t}M_{s}dM_{s}$ in $\mathbb{L}^{2}$ as $\mu(\pi_{t}%
^{n})\rightarrow0$. We call the limit the quadratic variation of $M$ and
denote it by $\langle M\rangle_{t}$.

It is clear that $\langle M\rangle$ is an increasing process. By Lemma
\ref{Mcontrol} and (\ref{M}), we can obtain for each $t,s\geq0$,
\[
\hat{\mathbb{E}}_{t}[|\langle M\rangle_{t+s}-\langle M\rangle_{t}%
|]=\hat{\mathbb{E}}_{t}[\langle M\rangle_{t+s}-\langle M\rangle_{t}%
]=\hat{\mathbb{E}}_{t}[|M_{t+s}-M_{t}|^{2}]\leq Cs.
\]

For symmetric martingales $M$ and $M^{\prime}$ satisfying (\ref{M}), we know
that $M+M^{\prime}$ and $M-M^{\prime}$ are also symmetric martingales
satisfying (\ref{M}). As%

\begin{align*}
&  \sum_{i=0}^{n-1}(M_{t_{i+1}^{n}}-M_{t_{i}^{n}})(M_{t_{i+1}^{n}}^{\prime
}-M_{t_{i}^{n}}^{\prime})\\
&  =\frac{1}{4}\sum_{i=0}^{n-1}\{[(M_{t_{i+1}^{n}}+M_{t_{i+1}^{n}}^{\prime
})-(M_{t_{i}^{n}}+M_{t_{i}^{n}}^{\prime})]^{2}-[(M_{t_{i+1}^{n}}%
-M_{t_{i+1}^{n}}^{\prime})-(M_{t_{i}^{n}}-M_{t_{i}^{n}}^{\prime})]^{2}\},
\end{align*}
we can define the mutual variation process by
\begin{align*}
\langle M,M^{\prime}\rangle_{t}  &  :=\mathbb{L}^{2}-\lim_{\mu(\pi_{t}%
^{n})\rightarrow0}\sum_{i=0}^{n-1}(M_{t_{i+1}^{n}}-M_{t_{i}^{n}}%
)(M_{t_{i+1}^{n}}^{\prime}-M_{t_{i}^{n}}^{\prime})\\
&  \ =\frac{1}{4}[\langle M+M^{\prime}\rangle_{t}-\langle M-M^{\prime}%
\rangle_{t}].
\end{align*}

For a $d$-dimensional symmetric martingale $M=(M^{1},\cdots,M^{d})^{T}$ with
$M^{i}$ satisfying (\ref{M}) for $i=1,\ldots,d$, we define the quadratic
variation by
\[
\langle M\rangle_{t}:=(\langle M^{i},M^{j}\rangle_{t})_{1\leq i,j\leq d}.
\]

\section{L\'{e}vy's martingale characterization of $G$-Brownian motion}

The following theorem is the L\'{e}vy's martingale characterization of
$G$-Brownian motion on the complete consistent sublinear expectation space.

\begin{theorem}
\label{levy theorem} Let $(\Omega,\mathbb{L}^{1}(\Omega),(\mathbb{L}%
^{1}(\Omega_{t}))_{t\geq0},(\hat{\mathbb{E}}_{t})_{t\geq0})$ be a complete
consistent sublinear expectation space and $G:\mathbb{S}(d)\rightarrow
\mathbb{R}$ be a given monotonic and sublinear function. Assume $(M_{t}%
)_{t\geq0}$ is a $d$-dimensional symmetric martingale satisfying
$M_{0}=0,M_{t}\in(\mathbb{L}^{3}(\Omega_{t}))^{d}$ for each $t\geq0$ and
$\sup \{ \hat{\mathbb{E}}[|M_{t+\delta}-M_{t}|^{3}]:t\leq T\}=o(\delta
)\ $as$\  \delta \downarrow0$ for each fixed $T>0$. Then the following
conditions are equivalent:

\begin{description}
\item[(1)] $(M_{t})_{t\geq0}$ is a $G$-Brownian motion;

\item[(2)] $\hat{\mathbb{E}}_{t}[|M_{t+s}-M_{t}|^{2}]\leq Cs$ for each
$t,s\geq0$ and the process $\frac{1}{2}$\textrm{tr}$[A\langle M\rangle
_{t}]-G(A)t$, $t\geq0$, is a martingale for each $A\in \mathbb{S}(d)$, where
$C>0$ is a constant;

\item[(3)] The process $\frac{1}{2}\langle AM_{t},M_{t}\rangle-G(A)t$,
$t\geq0$, is a martingale for each $A\in \mathbb{S}(d)$.
\end{description}
\end{theorem}

In order to prove the above theorem, we need the following lemmas. For
simplicity, we only prove the theorem for the case $d=1$, and the proof for
$d>1$ is the same. Under the case $d=1$,%
\begin{equation}
G(a)=\frac{1}{2}(\bar{\sigma}^{2}a^{+}-\underline{\sigma}^{2}a^{-})\text{ for
}a\in \mathbb{R}, \label{neweq-3}%
\end{equation}
where $0\leq \underline{\sigma}^{2}\leq \bar{\sigma}^{2}<\infty$. We suppose
$\bar{\sigma}^{2}>0$, since $M\equiv0$ if $\bar{\sigma}^{2}=0$, which is trivial.

The main idea to prove the above main theorem is to use PDE method, which
needs that the solution $u$ for $G$-heat equation (\ref{GPDE}) belongs to
$C^{1,2}$. But, when $\underline{\sigma}^{2}=0$, $u$ may not be in $C^{1,2}$.
So we introduce the following auxiliary space.

Let $(\bar{\Omega},\mathbb{L}^{1}(\bar{\Omega}),\mathbb{\bar{E}})$ be a
complete linear expectation space and $(W_{t})_{t\geq0}$ be a standard
$1$-dimensional classical Brownian motion on it. Define
\[
\tilde{\Omega}:=\{ \tilde{\omega}=(\omega,\bar{\omega}):\omega \in \Omega
,\bar{\omega}\in \bar{{\Omega}}\}.
\]
Let $T>0$ be fixed. For any given partition $\pi=\{t_{0,}\ldots,t_{n}\}$ of
$[0,T]$, i.e., $0=t_{0}<\cdots<t_{n}=T$, set $\mathcal{\tilde{H}}_{0}^{\pi
}:=\mathbb{R}$ and for $i=1,\ldots,n$,%
\begin{equation}
\mathcal{\tilde{H}}_{t_{i}}^{\pi}:=\{ \xi(\tilde{\omega})=\phi(\omega
,W_{t_{1}}(\bar{\omega}),\cdots,W_{t_{i}}(\bar{\omega})-W_{t_{i-1}}%
(\bar{\omega})):\tilde{\omega}=(\omega,\bar{\omega})\in \tilde{\Omega},\phi \in
L_{t_{i}}^{\pi}\}, \label{neweq-3-3}%
\end{equation}
where $L_{t_{i}}^{\pi}$ is the space of functions $\phi:\Omega \times
\mathbb{R}^{i}\rightarrow \mathbb{R}$ satisfying the following properties:

\begin{description}
\item[(1)] $\phi$ is bounded and $\phi(\cdot,x)\in \mathbb{L}^{1}(\Omega
_{t_{i}})$ for each $x\in \mathbb{R}^{i}$;

\item[(2)] There exists a constant $K>0$ such that $|\phi(\omega
,x)-\phi(\omega,x^{\prime})|\leq K|x-x^{\prime}|$ for each $\omega \in \Omega$,
$x,x^{\prime}\in \mathbb{R}^{i}$.
\end{description}

It is easy to verify that $\mathcal{\tilde{H}}_{t_{i}}^{\pi}\subset
\mathcal{\tilde{H}}_{t_{j}}^{\pi}$ for $0\leq i<j\leq n$. Now we want to
define a sublinear conditional expectation $\mathbb{\tilde{E}}_{t_{i}}^{\pi
}:\mathcal{\tilde{H}}_{t_{n}}^{\pi}\rightarrow \mathcal{\tilde{H}}_{t_{i}}%
^{\pi}$ such that $(\tilde{\Omega},(\mathcal{\tilde{H}}_{t_{i}}^{\pi}%
)_{i=0}^{n},(\mathbb{\tilde{E}}_{t_{i}}^{\pi})_{i=0}^{n})$ forms a discrete
consistent sublinear expectation space. For this purpose, we first define the
operator $\mathbb{\tilde{E}}_{t_{i},t_{i+1}}^{\pi}:\mathcal{\tilde{H}%
}_{t_{i+1}}^{\pi}\rightarrow \mathcal{\tilde{H}}_{t_{i}}^{\pi}$, $i=0,\ldots
,n-1$. For each $\xi(\tilde{\omega})=\phi(\omega,W_{t_{1}}(\bar{\omega
}),\cdots,W_{t_{i+1}}(\bar{\omega})-W_{t_{i}}(\bar{\omega}))\in \mathcal{\tilde
{H}}_{t_{i+1}}^{\pi}$, define%
\begin{equation}
\mathbb{\tilde{E}}_{t_{i},t_{i+1}}^{\pi}[\xi](\tilde{\omega}):=\varphi
(\omega,W_{t_{1}}(\bar{\omega}),\cdots,W_{t_{i}}(\bar{\omega})-W_{t_{i-1}%
}(\bar{\omega})), \label{neweq-3-4}%
\end{equation}
where%
\begin{equation}
\varphi(\cdot,x_{1},\cdots,x_{i}):=\hat{\mathbb{E}}_{t_{i}}[\psi(\cdot
,x_{1},\cdots,x_{i})] \label{neweq-4}%
\end{equation}
and%
\begin{equation}
\psi(\omega,x_{1},\cdots,x_{i})=\bar{\mathbb{E}}[\phi(\omega,x_{1}%
,\cdots,x_{i},W_{t_{i+1}}-W_{t_{i}})]\text{ for }\omega \in \Omega.
\label{neweq-5}%
\end{equation}
In the following, we will show that $\mathbb{\tilde{E}}_{t_{i},t_{i+1}}^{\pi
}[\xi]\in \mathcal{\tilde{H}}_{t_{i}}^{\pi}$. Now we define $\mathbb{\tilde{E}%
}_{t_{i}}^{\pi}:\mathcal{\tilde{H}}_{t_{n}}^{\pi}\rightarrow \mathcal{\tilde
{H}}_{t_{i}}^{\pi}$ backwardly from $i=n$ to $i=0$ as follows: for each
$X\in \mathcal{\tilde{H}}_{t_{n}}^{\pi}$, define%
\begin{equation}
\mathbb{\tilde{E}}_{t_{i}}^{\pi}[X]:=\mathbb{\tilde{E}}_{t_{i},t_{i+1}}^{\pi
}[\mathbb{\tilde{E}}_{t_{i+1}}^{\pi}[X]], \ i=0,\ldots
,n-1, \label{neweq-6}%
\end{equation}
where%
\[
\mathbb{\tilde{E}}_{t_{n}}^{\pi}[X]=X.
\]

\begin{lemma}
\label{newlem1} Let $(\mathcal{\tilde{H}}_{t_{i}}^{\pi})_{i=0}^{n}$ be defined
in (\ref{neweq-3-3}) and $(\mathbb{\tilde{E}}_{t_{i}}^{\pi})_{i=0}^{n}$ be
defined in (\ref{neweq-6}). Then $(\tilde{\Omega},(\mathcal{\tilde{H}}_{t_{i}%
}^{\pi})_{i=0}^{n},(\mathbb{\tilde{E}}_{t_{i}}^{\pi})_{i=0}^{n})$ forms a
discrete consistent sublinear expectation space. Furthermore, if $\xi
\in \mathcal{\tilde{H}}_{t_{n}}^{\pi}$ is independent of $\bar{\omega}$ (resp.
$\omega$), then $\mathbb{\tilde{E}}_{t_{i}}^{\pi}[\xi](\omega,\bar{\omega
})=\hat{\mathbb{E}}_{t_{i}}[\xi](\omega)$ (resp. $\mathbb{\tilde{E}}_{t_{i}%
}^{\pi}[\xi](\omega,\bar{\omega})=\bar{\mathbb{E}}_{t_{i}}[\xi](\bar{\omega})$).
\end{lemma}

\begin{proof}
We only prove that $\mathbb{\tilde{E}}_{t_{i},t_{i+1}}^{\pi}$ in
(\ref{neweq-3-4}) is well-defined, the other properties can be easily
verified. For this, we only need to show that $\psi(\cdot,x)\in \mathbb{L}%
^{1}(\Omega_{t_{i+1}})$ for each $x=(x_{1},\cdots,x_{i})\in \mathbb{R}^{i}$.
For each given integer $N>0$, set $x_{j}^{N}=-N+\frac{2j}{N}$, $j=0,\ldots
,N^{2}$. Then, for each $\omega \in \Omega$,%
\begin{align*}
&  \left \vert \bar{\mathbb{E}}[\phi(\omega,x,W_{t_{i+1}}-W_{t_{i}}%
)]-\bar{\mathbb{E}}\left[  \sum_{j=0}^{N^{2}-1}\phi(\omega,x,x_{j}%
^{N})I_{[x_{j}^{N},x_{j+1}^{N})}(W_{t_{i+1}}-W_{t_{i}})\right]  \right \vert \\
&  \leq \frac{2K}{N}+L\bar{\mathbb{E}}[I_{[N,\infty)}(|W_{t_{i+1}}-W_{t_{i}%
}|)]\\
&  \leq \frac{2K}{N}+\frac{L}{N}\bar{\mathbb{E}}[|W_{t_{i+1}}-W_{t_{i}}|],
\end{align*}
where $L>0$ such that $|\phi|\leq L$. From this, we can deduce that%
\begin{equation}
\left \vert \psi(\cdot,x)-\sum_{j=0}^{N^{2}-1}a_{j}^{N}\phi(\cdot,x,x_{j}%
^{N})\right \vert \leq \frac{2K}{N}+\frac{L}{N}\bar{\mathbb{E}}[|W_{t_{i+1}%
}-W_{t_{i}}|], \label{neweq-7}%
\end{equation}
where $a_{j}^{N}=\bar{\mathbb{E}}[I_{[x_{j}^{N},x_{j+1}^{N})}(W_{t_{i+1}%
}-W_{t_{i}})]$. Noting that $\phi(\cdot,x_{1},\cdots,x_{i},x_{j}^{N}%
)\in \mathbb{L}^{1}(\Omega_{t_{i+1}})$, we obtain $\psi(\cdot,x_{1}%
,\cdots,x_{i})\in \mathbb{L}^{1}(\Omega_{t_{i+1}})$ by (\ref{neweq-7}).
\end{proof}

The completion of $(\tilde{\Omega},(\mathcal{\tilde{H}}_{t_{i}}^{\pi}%
)_{i=0}^{n},(\mathbb{\tilde{E}}_{t_{i}}^{\pi})_{i=0}^{n})$ is denoted by
$(\tilde{\Omega},(\mathbb{L}^{1}(\tilde{\Omega}_{t_{i}}^{\pi}))_{i=0}%
^{n},(\mathbb{\tilde{E}}_{t_{i}}^{\pi})_{i=0}^{n})$.

\begin{lemma}
\label{newlem2} For each $\varepsilon \in(0,1)$, set%
\[
M_{t}^{\varepsilon}(\omega,\bar{\omega}):=M_{t}(\omega)+\varepsilon W_{t}%
(\bar{\omega})\text{ for }t\geq0,(\omega,\bar{\omega})\in \tilde{\Omega
}\text{.}%
\]
If $\frac{1}{2}a(M_{t})^{2}-G(a)t$, $t\geq0$, is a martingale for each
$a\in \mathbb{R}$, then, under $(\tilde{\Omega},(\mathbb{L}^{1}(\tilde{\Omega
}_{t_{i}}^{\pi}))_{i=0}^{n},(\mathbb{\tilde{E}}_{t_{i}}^{\pi})_{i=0}^{n})$, we have

\begin{description}
\item[(1)] $(M_{t_{i}}^{\varepsilon})_{i=0}^{n}$ is a discrete symmetric martingale;

\item[(2)] $\left(  \frac{1}{2}a(M_{t_{i}}^{\varepsilon})^{2}-G_{\varepsilon
}(a)t_{i}\right)  _{i=0}^{n}$ is a discrete martingale, where $G_{\varepsilon
}(a)=G(a)+\frac{1}{2}\varepsilon^{2}a$ for $a\in \mathbb{R}$.
\end{description}
\end{lemma}

\begin{proof}
(1) By the definition of $\mathbb{\tilde{E}}_{t_{i}}^{\pi}$, we have, for
$i=0,\ldots,n-1$,%
\[
\mathbb{\tilde{E}}_{t_{i}}^{\pi}[M_{t_{i+1}}^{\varepsilon}](\omega,\bar
{\omega})=\hat{\mathbb{E}}_{t_{i}}[M_{t_{i+1}}](\omega)+\varepsilon
\bar{\mathbb{E}}[W_{t_{i+1}}-W_{t_{i}}]+\varepsilon W_{t_{i}}(\bar{\omega
})=M_{t_{i}}^{\varepsilon}(\omega,\bar{\omega})
\]
and%
\[
\mathbb{\tilde{E}}_{t_{i}}^{\pi}[-M_{t_{i+1}}^{\varepsilon}](\omega
,\bar{\omega})=\hat{\mathbb{E}}_{t_{i}}[-M_{t_{i+1}}](\omega)+\varepsilon
\bar{\mathbb{E}}[-(W_{t_{i+1}}-W_{t_{i}})]-\varepsilon W_{t_{i}}(\bar{\omega
})=-M_{t_{i}}^{\varepsilon}(\omega,\bar{\omega}).
\]
Thus $(M_{t_{i}}^{\varepsilon})_{i=0}^{n}$ is a discrete symmetric martingale.

(2) Since $\frac{1}{2}a(M_{t})^{2}-G(a)t$, $t\geq0$, is a martingale, we have%
\[
\hat{\mathbb{E}}_{t_{i}}[a(M_{t_{i+1}})^{2}-a(M_{t_{i}})^{2}]=2G(a)(t_{i+1}%
-t_{i}).
\]
By the fact that $(M_{t})_{t\geq0}$ is a symmetric martingale, we also get%
\[
\hat{\mathbb{E}}_{t_{i}}[a(M_{t_{i+1}})^{2}-a(M_{t_{i}})^{2}]=\hat{\mathbb{E}%
}_{t_{i}}[a(M_{t_{i+1}}-M_{t_{i}})^{2}].
\]
Thus%
\[
\hat{\mathbb{E}}_{t_{i}}[a(M_{t_{i+1}}-M_{t_{i}})^{2}]=2G(a)(t_{i+1}-t_{i}).
\]
It follows from (1) and the definition of $\mathbb{\tilde{E}}_{t_{i}}^{\pi}$
that%
\[
\mathbb{\tilde{E}}_{t_{i}}^{\pi}[M_{t_{i}}^{\varepsilon}(M_{t_{i+1}%
}^{\varepsilon}-M_{t_{i}}^{\varepsilon})]=\mathbb{\tilde{E}}_{t_{i}}^{\pi
}[-M_{t_{i}}^{\varepsilon}(M_{t_{i+1}}^{\varepsilon}-M_{t_{i}}^{\varepsilon
})]=0
\]
and%
\begin{align*}
\mathbb{\tilde{E}}_{t_{i}}^{\pi}[a(M_{t_{i+1}}^{\varepsilon}-M_{t_{i}%
}^{\varepsilon})^{2}](\omega,\bar{\omega})  &  =\hat{\mathbb{E}}_{t_{i}%
}[a(M_{t_{i+1}}-M_{t_{i}})^{2}](\omega)+\varepsilon^{2}\bar{\mathbb{E}%
}[(W_{t_{i+1}}-W_{t_{i}})^{2}]\\
&  =2G(a)(t_{i+1}-t_{i})+\varepsilon^{2}(t_{i+1}-t_{i}).
\end{align*}
Then%
\[
\mathbb{\tilde{E}}_{t_{i}}^{\pi}[a(M_{t_{i+1}}^{\varepsilon})^{2}]=a(M_{t_{i}%
}^{\varepsilon})^{2}+\mathbb{\tilde{E}}_{t_{i}}^{\pi}[a(M_{t_{i+1}%
}^{\varepsilon}-M_{t_{i}}^{\varepsilon})^{2}]=a(M_{t_{i}}^{\varepsilon}%
)^{2}+2G_{\varepsilon}(a)(t_{i+1}-t_{i}),
\]
which implies that $\left(  \frac{1}{2}a(M_{t_{i}}^{\varepsilon}%
)^{2}-G_{\varepsilon}(a)t_{i}\right)  _{i=0}^{n}$ is a discrete martingale.
\end{proof}

\begin{remark}
When $\underline{\sigma}^{2}=0$, $G_{\varepsilon}(a)=\frac{1}{2}[(\bar{\sigma
}^{2}+\varepsilon^{2})a^{+}-\varepsilon^{2}a^{-}]$ is non-degenerate. So we
can consider $M^{\varepsilon}$ in Theorem \ref{levy theorem}.
\end{remark}

Consider the following two PDEs:%
\begin{equation}
\partial_{t}u+G(\partial_{xx}^{2}u)=0,\ u(T,x)=\varphi(x), \label{neweq-8}%
\end{equation}%
\begin{equation}
\partial_{t}u^{\varepsilon}+G_{\varepsilon}(\partial_{xx}^{2}u^{\varepsilon
})=0,\ u^{\varepsilon}(T,x)=\varphi(x). \label{neweq-9}%
\end{equation}

\begin{lemma}
\label{newlem3} For each $\varepsilon \in(0,1)$ and $\varphi \in C_{b.Lip}%
(\mathbb{R})$, let $u$ and $u^{\varepsilon}$ be the solution of PDEs
(\ref{neweq-8}) and (\ref{neweq-9}) respectively. Then, for $t$, $t^{\prime
}\in \lbrack0,T]$, $x\in \mathbb{R}$,%
\[
|u^{\varepsilon}(t,x)-u^{\varepsilon}(t^{\prime},x)|\leq C_{\varphi}%
\sqrt{2(\bar{\sigma}^{2}+1)/\pi}\sqrt{|t-t^{\prime}|}%
\]
and%
\[
|u^{\varepsilon}(t,x)-u(t,x)|\leq C_{\varphi}\sqrt{2(T-t)/\pi}\varepsilon,
\]
where $C_{\varphi}$ is the Lipschitz constant of $\varphi$.
\end{lemma}

\begin{proof}
Let $(B_{t}^{1},B_{t}^{2})_{t\geq0}$ be a $2$-dimensional $\tilde{G}$-Brownian
motion on a complete sublinear expectation space $(\Omega,\mathbb{L}%
^{1}(\Omega),\hat{\mathbb{E}})$ with
\[
\tilde{G}\left(  \left[
\begin{array}
[c]{cc}%
a & b\\
b & c
\end{array}
\right]  \right)  =G(a)+\frac{1}{2}c\text{ for }a,b,c\in \mathbb{R}.
\]
It is easy to verify that $(B_{t}^{1})_{t\geq0}$ is a $1$-dimensional
$G$-Brownian motion and $(B_{t}^{1}+\varepsilon B_{t}^{2})_{t\geq0}$ is a
$1$-dimensional $G_{\varepsilon}$-Brownian motion. By Remark \ref{newre-000},
we deduce that, for $(t,x)\in \lbrack0,T]\times \mathbb{R}$,%
\[
u(t,x)=\hat{\mathbb{E}}[\varphi(x+B_{T}^{1}-B_{t}^{1})]\text{ and
}u^{\varepsilon}(t,x)=\hat{\mathbb{E}}[\varphi(x+B_{T}^{1}+\varepsilon
B_{T}^{2}-B_{t}^{1}-\varepsilon B_{t}^{2})].
\]
Thus, by Proposition 1.5 in Chapter III of \cite{Peng 1}, we obtain%
\[
|u^{\varepsilon}(t,x)-u^{\varepsilon}(t^{\prime},x)|\leq C_{\varphi}%
\hat{\mathbb{E}}[|B_{t}^{1}+\varepsilon B_{t}^{2}-B_{t^{\prime}}%
^{1}-\varepsilon B_{t^{\prime}}^{2}|]=C_{\varphi}\sqrt{2(\bar{\sigma}%
^{2}+\varepsilon^{2})/\pi}\sqrt{|t-t^{\prime}|}%
\]
and%
\[
|u^{\varepsilon}(t,x)-u(t,x)|\leq C_{\varphi}\varepsilon \hat{\mathbb{E}%
}[|B_{T}^{2}-B_{t}^{2}|]=C_{\varphi}\sqrt{2(T-t)/\pi}\varepsilon.
\]

\end{proof}

\textbf{Proof of Theorem \ref{levy theorem}.} (1)$\Longrightarrow$(2) If
$(M_{t})_{t\geq0}$ is a $G$-Brownian motion, then, by Proposition 1.4 in
Chapter IV of \cite{Peng 1}, we get $\hat{\mathbb{E}}_{t}[|M_{t+s}-M_{t}%
|^{2}]=\bar{\sigma}^{2}s$, $\langle M\rangle_{t}-\bar{\sigma}^{2}t$ and
$\underline{\sigma}^{2}t-\langle M\rangle_{t}$, $t\geq0$, are two martingales,
which implies that $\frac{1}{2}a\langle M\rangle_{t}-G(a)t$, $t\geq0$, is a
martingale for each $a\in \mathbb{R}$.

(2)$\Longrightarrow$(3) Note that $M_{t}^{2}=2\int_{0}^{t}M_{s}dM_{s}+\langle
M\rangle_{t}$ and $(\int_{0}^{t}M_{s}dM_{s})_{t\geq0}$ is a symmetric
martingale, then $\frac{1}{2}a(M_{t})^{2}-G(a)t$, $t\geq0$, is a martingale
for each $a\in \mathbb{R}$.

(3)$\Longrightarrow$(1) By Proposition \ref{newpro01}, we only need to prove
that%
\[
\hat{\mathbb{E}}_{t}[\varphi(M_{T}-M_{t})]=u(t,0),
\]
where $0<t<T$, $\varphi \in C_{b.Lip}(\mathbb{R})$ and $u$ is the solution of
PDE (\ref{neweq-8}). The proof is divided into three steps.

\textbf{Step 1: Taylor's expansion.} For each fixed $\varepsilon \in(0,1)$ and
$h\in(0,t)$, let $v^{\varepsilon}$ be the solution of the following PDE:%
\begin{equation}
\partial_{t}v^{\varepsilon}+G_{\varepsilon}(\partial_{xx}^{2}v^{\varepsilon
})=0,\text{ }v^{\varepsilon}(T+h,x)=\varphi(x). \label{neweq-10}%
\end{equation}
It is clear that $v^{\varepsilon}(s,x)=u^{\varepsilon}(s-h,x)$ for
$(s,x)\in \lbrack h,T+h]\times \mathbb{R}$, where $u^{\varepsilon}$ is the
solution of PDE (\ref{neweq-9}). Since PDE (\ref{neweq-10}) is uniformly
parabolic and $G_{\varepsilon}(\cdot)$ is a convex function, by the interior
regularity of $v^{\varepsilon}$ (see \cite{Krylov 1, W-L 1}),%
\begin{equation}
||v^{\varepsilon}||_{C^{1+\alpha/2,2+\alpha}([0,T]\times \mathbb{R})}%
<\infty \text{ for some }\alpha \in(0,1). \label{neweq-11}%
\end{equation}
For each $n\geq1$, set $t_{i}^{n}=t+n^{-1}i(T-t)$, $i=0,\ldots,n$. Let
$(M_{t}^{\varepsilon})_{t\geq0}$ be defined as in Lemma \ref{newlem2}. Denote
$M_{t_{2}}^{\varepsilon,t_{1}}:=M_{t_{2}}^{\varepsilon}-M_{t_{1}}%
^{\varepsilon}$ for $0\leq t_{1}\leq t_{2}$, then%
\begin{equation}%
\begin{array}
[c]{rl}%
\displaystyle v^{\varepsilon}(T,M_{T}^{\varepsilon,t})-v^{\varepsilon}(t,0) &
\displaystyle=\sum_{i=0}^{n-1}[v^{\varepsilon}(t_{i+1}^{n},M_{t_{i+1}^{n}%
}^{\varepsilon,t})-v^{\varepsilon}(t_{i}^{n},M_{t_{i}^{n}}^{\varepsilon,t})]\\
& \displaystyle=\sum_{i=0}^{n-1}[J_{i}^{n}+I_{i}^{n}],
\end{array}
\label{neweq-12}%
\end{equation}
where%
\[
J_{i}^{n}=\partial_{t}v^{\varepsilon}(t_{i}^{n},M_{t_{i}^{n}}^{\varepsilon
,t})(t_{i+1}^{n}-t_{i}^{n})+\partial_{x}v^{\varepsilon}(t_{i}^{n},M_{t_{i}%
^{n}}^{\varepsilon,t})M_{t_{i+1}^{n}}^{\varepsilon,t_{i}^{n}}+\frac{1}%
{2}\partial_{xx}^{2}v^{\varepsilon}(t_{i}^{n},M_{t_{i}^{n}}^{\varepsilon
,t})\left(  M_{t_{i+1}^{n}}^{\varepsilon,t_{i}^{n}}\right)  ^{2}%
\]
and%
\[%
\begin{array}
[c]{rl}%
\displaystyle I_{i}^{n}= & \displaystyle(t_{i+1}^{n}-t_{i}^{n})\int_{0}%
^{1}[\partial_{t}v^{\varepsilon}(t_{i}^{n}+\alpha(t_{i+1}^{n}-t_{i}%
^{n}),M_{t_{i+1}^{n}}^{\varepsilon,t})-\partial_{t}v^{\varepsilon}(t_{i}%
^{n},M_{t_{i}^{n}}^{\varepsilon,t})]d\alpha \\
& \displaystyle+(M_{t_{i+1}^{n}}^{\varepsilon,t_{i}^{n}})^{2}\int_{0}^{1}%
\int_{0}^{1}\alpha \lbrack \partial_{xx}^{2}v^{\varepsilon}(t_{i}^{n}%
,M_{t_{i}^{n}}^{\varepsilon,t}+\alpha \beta M_{t_{i+1}^{n}}^{\varepsilon
,t_{i}^{n}})-\partial_{xx}^{2}v^{\varepsilon}(t_{i}^{n},M_{t_{i}^{n}%
}^{\varepsilon,t})]d\beta d\alpha.
\end{array}
\]

\textbf{Step 2: Estimation of residual terms.} Set $\pi_{n}=\{0,t_{0}%
^{n},t_{1}^{n},\ldots,t_{n}^{n}\}$, by (\ref{neweq-12}) and Lemma
\ref{newlem1}, we get%
\begin{equation}
\left \vert \mathbb{\tilde{E}}_{t}^{\pi_{n}}[v^{\varepsilon}(T,M_{T}%
^{\varepsilon,t})]-v^{\varepsilon}(t,0)-\mathbb{\tilde{E}}_{t}^{\pi_{n}%
}\left[  \sum_{i=0}^{n-1}J_{i}^{n}\right]  \right \vert \leq \mathbb{\tilde{E}%
}_{t}^{\pi_{n}}\left[  \left \vert \sum_{i=0}^{n-1}I_{i}^{n}\right \vert
\right]  . \label{neweq-13}%
\end{equation}
It follows from Proposition \ref{newpro01} and Lemma \ref{newlem2} that%
\begin{align*}
&  \mathbb{\tilde{E}}_{t_{i}^{n}}^{\pi_{n}}[J_{i}^{n}]\\
&  =\mathbb{\tilde{E}}_{t_{i}^{n}}^{\pi_{n}}[\partial_{t}v^{\varepsilon}%
(t_{i}^{n},x)(t_{i+1}^{n}-t_{i}^{n})+\partial_{x}v^{\varepsilon}(t_{i}%
^{n},x)M_{t_{i+1}^{n}}^{\varepsilon,t_{i}^{n}}+\frac{1}{2}\partial_{xx}%
^{2}v^{\varepsilon}(t_{i}^{n},x)(M_{t_{i+1}^{n}}^{\varepsilon,t_{i}^{n}}%
)^{2}]_{x=M_{t_{i}^{n}}^{\varepsilon,t}}\\
&  =[\partial_{t}v^{\varepsilon}(t_{i}^{n},M_{t_{i}^{n}}^{\varepsilon
,t})+G_{\varepsilon}(\partial_{xx}^{2}v^{\varepsilon}(t_{i}^{n},M_{t_{i}^{n}%
}^{\varepsilon,t}))](t_{i+1}^{n}-t_{i}^{n}).
\end{align*}
Since $v^{\varepsilon}$ satisfies PDE (\ref{neweq-10}), we obtain
$\mathbb{\tilde{E}}_{t_{i}^{n}}^{\pi_{n}}[J_{i}^{n}]=0$. Thus%
\begin{equation}
\mathbb{\tilde{E}}_{t}^{\pi_{n}}\left[  \sum_{i=0}^{n-1}J_{i}^{n}\right]
=\mathbb{\tilde{E}}_{t}^{\pi_{n}}\left[  \sum_{i=0}^{n-2}J_{i}^{n}%
+\mathbb{\tilde{E}}_{t_{n-1}^{n}}^{\pi_{n}}[J_{n-1}^{n}]\right]
=\mathbb{\tilde{E}}_{t}^{\pi_{n}}\left[  \sum_{i=0}^{n-2}J_{i}^{n}\right]
=\cdots=0. \label{neweq-14}%
\end{equation}
Combining (\ref{neweq-13}) with (\ref{neweq-14}), we conclude that%
\begin{equation}
\mathbb{\tilde{E}}^{\pi_{n}}\left[  \left \vert \mathbb{\tilde{E}}_{t}^{\pi
_{n}}[v^{\varepsilon}(T,M_{T}^{\varepsilon,t})]-v^{\varepsilon}%
(t,0)\right \vert \right]  \leq \sum_{i=0}^{n-1}\mathbb{\tilde{E}}^{\pi_{n}%
}\left[  \left \vert I_{i}^{n}\right \vert \right]  . \label{neweq-16}%
\end{equation}
As $v^{\varepsilon}$ satisfying (\ref{neweq-11}), we can easily get%
\[
|I_{i}^{n}|\leq C_{1}\left \{  \left \vert t_{i+1}^{n}-t_{i}^{n}\right \vert
^{1+\alpha/2}+\left \vert M_{t_{i+1}^{n}}^{\varepsilon,t_{i}^{n}}\right \vert
^{2+\alpha}\right \}  ,
\]
where $C_{1}>0$ is a constant depending on $\varepsilon$, $h$, $T$, $G$ and
$\varphi$. By Lemma \ref{newlem1}, we have%
\begin{align*}
\mathbb{\tilde{E}}^{\pi_{n}}\left[  \left \vert M_{t_{i+1}^{n}}^{\varepsilon
,t_{i}^{n}}\right \vert ^{2+\alpha}\right]   &  \leq2^{1+\alpha}\left \{
\mathbb{\tilde{E}}^{\pi_{n}}\left[  \left \vert M_{t_{i+1}^{n}}-M_{t_{i}^{n}%
}\right \vert ^{2+\alpha}\right]  +\mathbb{\tilde{E}}^{\pi_{n}}\left[
\left \vert W_{t_{i+1}^{n}}-W_{t_{i}^{n}}\right \vert ^{2+\alpha}\right]
\right \} \\
&  \leq2^{1+\alpha}\left \{  \hat{\mathbb{E}}\left[  \left \vert M_{t_{i+1}^{n}%
}-M_{t_{i}^{n}}\right \vert ^{2+\alpha}\right]  +\bar{\mathbb{E}}\left[
\left \vert W_{1}\right \vert ^{2+\alpha}\right]  \left \vert t_{i+1}^{n}%
-t_{i}^{n}\right \vert ^{1+\alpha/2}\right \}  .
\end{align*}
Thus we get%
\begin{equation}
\mathbb{\tilde{E}}^{\pi_{n}}\left[  \left \vert I_{i}^{n}\right \vert \right]
\leq C_{2}\left \{  \left \vert t_{i+1}^{n}-t_{i}^{n}\right \vert ^{1+\alpha
/2}+\hat{\mathbb{E}}\left[  \left \vert M_{t_{i+1}^{n}}-M_{t_{i}^{n}%
}\right \vert ^{2+\alpha}\right]  \right \}  , \label{neweq-17}%
\end{equation}
where $C_{2}>0$ is a constant depending on $\varepsilon$, $h$, $T$, $G$ and
$\varphi$. By H\"{o}lder's inequality, we obtain%
\begin{equation}
\hat{\mathbb{E}}\left[  \left \vert M_{t_{i+1}^{n}}-M_{t_{i}^{n}}\right \vert
^{2+\alpha}\right]  \leq \left(  \hat{\mathbb{E}}\left[  \left \vert
M_{t_{i+1}^{n}}-M_{t_{i}^{n}}\right \vert ^{2}\right]  \right)  ^{1-\alpha
}\left(  \hat{\mathbb{E}}\left[  \left \vert M_{t_{i+1}^{n}}-M_{t_{i}^{n}%
}\right \vert ^{3}\right]  \right)  ^{\alpha}. \label{neweq-18}%
\end{equation}
Note that $\sup \{ \hat{\mathbb{E}}[|M_{t+\delta}-M_{t}|^{3}]:t\leq
T\}=o(\delta)\ $as$\  \delta \downarrow0$, $\hat{\mathbb{E}}[|M_{t_{i+1}^{n}%
}-M_{t_{i}^{n}}|^{2}]=2G(1)(t_{i+1}^{n}-t_{i}^{n})$ and $t_{i+1}^{n}-t_{i}%
^{n}=n^{-1}(T-t)$, then, by (\ref{neweq-17}) and (\ref{neweq-18}), we get%
\begin{equation}
\sum_{i=0}^{n-1}\mathbb{\tilde{E}}^{\pi_{n}}\left[  \left \vert I_{i}%
^{n}\right \vert \right]  =o(1). \label{new-neweq19}%
\end{equation}

\textbf{Step 3: Approximation.} From (\ref{neweq-16}) and (\ref{new-neweq19}),
we obtain
\begin{equation}
\lim_{n\rightarrow \infty}\mathbb{\tilde{E}}^{\pi_{n}}\left[  \left \vert
\mathbb{\tilde{E}}_{t}^{\pi_{n}}[v^{\varepsilon}(T,M_{T}^{\varepsilon
,t})]-v^{\varepsilon}(t,0)\right \vert \right]  =0. \label{neweq-19}%
\end{equation}
It follows from Lemma \ref{newlem3} that%
\[
|v^{\varepsilon}(T,x)-v^{\varepsilon}(T+h,x)|=|u^{\varepsilon}%
(T-h,x)-u^{\varepsilon}(T,x)|\leq C_{\varphi}\sqrt{2(\bar{\sigma}^{2}+1)/\pi
}\sqrt{h}%
\]
and%
\begin{align*}
|v^{\varepsilon}(t,0)-u(t,0)|  &  =|u^{\varepsilon}(t-h,0)-u(t,0)|\\
&  \leq|u^{\varepsilon}(t-h,0)-u^{\varepsilon}(t,0)|+|u^{\varepsilon
}(t,0)-u(t,0)|\\
&  \leq C_{\varphi}\sqrt{2(\bar{\sigma}^{2}+1)/\pi}\sqrt{h}+C_{\varphi}%
\sqrt{2(T-t)/\pi}\varepsilon.
\end{align*}
Since $v^{\varepsilon}(T+h,x)=\varphi(x)$, by Lemma \ref{newlem1}, we deduce%
\begin{align*}
&  \hat{\mathbb{E}}\left[  \left \vert \hat{\mathbb{E}}_{t}[\varphi(M_{T}%
-M_{t})]-u(t,0)\right \vert \right] \\
&  =\mathbb{\tilde{E}}^{\pi_{n}}\left[  \left \vert \mathbb{\tilde{E}}_{t}%
^{\pi_{n}}[\varphi(M_{T}-M_{t})]-u(t,0)\right \vert \right] \\
&  \leq \mathbb{\tilde{E}}^{\pi_{n}}\left[  |\varphi(M_{T}-M_{t})]-\varphi
(M_{T}^{\varepsilon,t})|+|v^{\varepsilon}(T+h,M_{T}^{\varepsilon
,t})-v^{\varepsilon}(T,M_{T}^{\varepsilon,t})|\right] \\
&  \  \  \ +\mathbb{\tilde{E}}^{\pi_{n}}\left[  \left \vert \mathbb{\tilde{E}%
}_{t}^{\pi_{n}}[v^{\varepsilon}(T,M_{T}^{\varepsilon,t})]-v^{\varepsilon
}(t,0)\right \vert \right]  +|v^{\varepsilon}(t,0)-u(t,0)|\\
&  \leq \mathbb{\tilde{E}}^{\pi_{n}}\left[  \left \vert \mathbb{\tilde{E}}%
_{t}^{\pi_{n}}[v^{\varepsilon}(T,M_{T}^{\varepsilon,t})]-v^{\varepsilon
}(t,0)\right \vert \right]  +C_{\varphi}\varepsilon \mathbb{\bar{E}}%
[|W_{T}-W_{t}|]+2C_{\varphi}\sqrt{2(\bar{\sigma}^{2}+1)/\pi}\sqrt
{h}+C_{\varphi}\sqrt{2(T-t)/\pi}\varepsilon \\
&  \leq \mathbb{\tilde{E}}^{\pi_{n}}\left[  \left \vert \mathbb{\tilde{E}}%
_{t}^{\pi_{n}}[v^{\varepsilon}(T,M_{T}^{\varepsilon,t})]-v^{\varepsilon
}(t,0)\right \vert \right]  +2C_{\varphi}\sqrt{2(\bar{\sigma}^{2}+1)/\pi}%
\sqrt{h}+2C_{\varphi}\sqrt{2(T-t)/\pi}\varepsilon.
\end{align*}
Taking $n\rightarrow \infty$, by (\ref{neweq-19}), we get%
\[
\hat{\mathbb{E}}\left[  \left \vert \hat{\mathbb{E}}_{t}[\varphi(M_{T}%
-M_{t})]-u(t,0)\right \vert \right]  \leq2C_{\varphi}\sqrt{2(\bar{\sigma}%
^{2}+1)/\pi}\sqrt{h}+2C_{\varphi}\sqrt{2(T-t)/\pi}\varepsilon.
\]
Letting $h\rightarrow0$ and $\varepsilon \rightarrow0$, we obtain
$\hat{\mathbb{E}}_{t}[\varphi(M_{T}-M_{t})]=u(t,0)$. $\Box$

We now consider the L\'{e}vy's martingale characterization of $G$-Brownian
motion on the $G$-expectation space. In this case, we do not need the
assumptions $M_{t}\in(\mathbb{L}^{3}(\Omega_{t}))^{d}$ and $\sup \{
\hat{\mathbb{E}}[|M_{t+\delta}-M_{t}|^{3}]:t\leq T\}=o(\delta)\ $%
as$\  \delta \downarrow0$ as in Theorem \ref{levy theorem}.

\begin{theorem}
\label{levyG}Let $\bar{G}:\mathbb{S}(d^{\prime})\rightarrow \mathbb{R}$ and
$G:\mathbb{S}(d)\rightarrow \mathbb{R}$ be two given monotonic and sublinear
functions, and let $(\Omega,L_{\bar{G}}^{1}(\Omega),(L_{\bar{G}}^{1}%
(\Omega_{t}))_{t\geq0},(\hat{\mathbb{E}}_{t})_{t\geq0})$ be a $\bar{G}%
$-expectation space. Assume $(M_{t})_{t\geq0}$ is a $d$-dimensional symmetric
martingale satisfying $M_{0}=0,M_{t}\in(L_{\bar{G}}^{2}(\Omega_{t}))^{d}$ for
each $t\geq0$. Then the following conditions are equivalent:

\begin{description}
\item[(1)] $(M_{t})_{t\geq0}$ is a $G$-Brownian motion;

\item[(2)] The process $\frac{1}{2}\langle AM_{t},M_{t}\rangle-G(A)t$,
$t\geq0$, is a martingale for each $A\in \mathbb{S}(d)$;

\item[(3)] $\hat{\mathbb{E}}_{t}[|M_{t+s}-M_{t}|^{2}]\leq Cs$ for each
$t,s\geq0$ and the process $\frac{1}{2}$\textrm{tr}$[A\langle M\rangle
_{t}]-G(A)t$, $t\geq0$, is a martingale for each $A\in \mathbb{S}(d)$, where
$C>0$ is a constant.
\end{description}
\end{theorem}

\begin{proof}
We only prove the case $d=1$. The case $d>1$ is similar. The proof for
(1)$\Longrightarrow$(2) is the same as Theorem \ref{levy theorem}.

(2)$\Longrightarrow$(3) Taking $A=1$, we get $M_{t}^{2}-\bar{\sigma}^{2}t$ is
a martingale, where $G(a)=\frac{1}{2}(\bar{\sigma}^{2}a^{+}-\underline{\sigma
}^{2}a^{-})$ for $a\in \mathbb{R}$. From this, we have%
\[
\hat{\mathbb{E}}_{t}[|M_{t+s}-M_{t}|^{2}]=\hat{\mathbb{E}}_{t}[M_{t+s}%
^{2}-M_{t}^{2}-2M_{t}(M_{t+s}-M_{t})]=\bar{\sigma}^{2}(t-s).
\]
Noting that $\langle M\rangle_{t}=M_{t}^{2}-2\int_{0}^{t}M_{s}dM_{s}$, we
obtain the desired result.

(3)$\Longrightarrow$(1) By Proposition \ref{newpro01}, we only need to prove
that%
\[
\hat{\mathbb{E}}_{t}[\varphi(M_{T}-M_{t})]=u(t,0),
\]
where $0<t<T$, $\varphi \in C_{b.Lip}(\mathbb{R})$ and $u$ is the solution of
PDE (\ref{neweq-8}).

Let $(B_{t})_{t\geq0}$ be the $\bar{G}$-Brownian motion. Following Section 2
in Chapter III in \cite{Peng 1}, we can construct an auxiliary $\tilde{G}%
$-expectation space $(\tilde{\Omega},L_{\tilde{G}}^{1}(\tilde{\Omega
}),(L_{\tilde{G}}^{1}(\tilde{\Omega}_{t}))_{t\geq0},(\mathbb{\tilde{E}}%
_{t})_{t\geq0})$ such that

\begin{description}
\item[(i)] $\tilde{\Omega}=\Omega \times C_{0}([0,\infty))$, where
$C_{0}([0,\infty))$ is the space of real-valued continuous paths $(\omega
_{t})_{t\geq0}$ with $\omega_{0}=0$;

\item[(ii)] $(B_{t},\bar{B}_{t})_{t\geq0}$ is a $\tilde{G}$-Brownian motion,
where $\bar{B}$ is the canonical process on $C_{0}([0,\infty))$, and
\[
\tilde{G}(A)=\bar{G}(A^{\prime})+\frac{1}{2}c\text{ for}\ A=\left[
\begin{array}
[c]{cc}%
A^{\prime} & b\\
b & c
\end{array}
\right]  \in \mathbb{S}(d^{\prime}+1).
\]

\end{description}

For each fixed $\varepsilon \in(0,1)$, define $M_{t}^{\varepsilon}%
=M_{t}+\varepsilon \bar{B}_{t}$. One can easily check that $(M_{t}%
^{\varepsilon})_{t\geq0}$ is a symmetric martingale. By Corollary 5.7 in
Chapter III in \cite{Peng 1}, we can deduce that $\frac{1}{2}a\langle
M^{\varepsilon}\rangle_{t}-G_{\varepsilon}(a)t$, $t\geq0$, is a martingale.
For each fixed $h\in(0,t)$, let $v^{\varepsilon}$ be the solution to the
following PDE:
\[
\partial_{t}v^{\varepsilon}+G_{\varepsilon}(\partial_{xx}^{2}v^{\varepsilon
})=0,\text{ }v^{\varepsilon}(T+h,x)=\varphi(x).
\]
By the interior regularity of $v^{\varepsilon}$ (see \cite{Krylov 1, W-L 1}),%
\[
||v^{\varepsilon}||_{C^{1+\alpha/2,2+\alpha}([0,T]\times \mathbb{R})}%
<\infty \text{ for some }\alpha \in(0,1).
\]
By martingale representation theorem for symmetric martingale (see Theorem 4.8 in
\cite{Song1}), applying It\^{o}'s formula to $v^{\varepsilon}(s,M_{s}%
^{\varepsilon,t})$ on $[t,T]$, where $M_{s}^{\varepsilon,t}=M_{s}%
^{\varepsilon}-M_{t}^{\varepsilon}$, we get%
\begin{align*}
v^{\varepsilon}(T,M_{T}^{\varepsilon,t})  &  =v^{\varepsilon}(t,0)+\int
_{t}^{T}\partial_{x}v^{\varepsilon}(s,M_{s}^{\varepsilon,t})dM_{s}%
^{\varepsilon}+\frac{1}{2}\int_{t}^{T}\partial_{xx}^{2}v^{\varepsilon}%
(s,M_{s}^{\varepsilon,t})d\langle M^{\varepsilon}\rangle_{s}+\int_{t}%
^{T}\partial_{t}v^{\varepsilon}(s,M_{s}^{\varepsilon,t})ds\\
&  =v^{\varepsilon}(t,0)+\int_{t}^{T}\partial_{x}v^{\varepsilon}%
(s,M_{s}^{\varepsilon,t})dM_{s}^{\varepsilon}+\frac{1}{2}\int_{t}^{T}%
\partial_{xx}^{2}v^{\varepsilon}(s,M_{s}^{\varepsilon,t})d\langle
M^{\varepsilon}\rangle_{s}-\int_{t}^{T}G_{\varepsilon}(\partial_{xx}%
^{2}v^{\varepsilon}(s,M_{s}^{\varepsilon,t}))ds.
\end{align*}
Taking conditional expectation $\tilde{\mathbb{E}}_{t}$ on both sides, we
have
\[
\tilde{\mathbb{E}}_{t}[v^{\varepsilon}(T,M_{T}^{\varepsilon,t}%
)]=v^{\varepsilon}(t,0)+\tilde{\mathbb{E}}_{t}\left[  \frac{1}{2}\int_{t}%
^{T}\partial_{xx}^{2}v^{\varepsilon}(s,M_{s}^{\varepsilon,t})d\langle
M^{\varepsilon}\rangle_{s}-\int_{t}^{T}G_{\varepsilon}(\partial_{xx}%
^{2}v^{\varepsilon}(s,M_{s}^{\varepsilon,t}))ds\right]  .
\]
Noting that $\frac{1}{2}a\langle M^{\varepsilon}\rangle_{t}-G_{\varepsilon
}(a)t$, $t\geq0$, is a martingale, by Proposition 1.4 in Chapter IV in
\cite{Peng 1}, we know that%
\[
\tilde{\mathbb{E}}_{t}\left[  \frac{1}{2}\int_{t}^{T}\partial_{xx}%
^{2}v^{\varepsilon}(s,M_{s}^{\varepsilon,t})d\langle M^{\varepsilon}%
\rangle_{s}-\int_{t}^{T}G_{\varepsilon}(\partial_{xx}^{2}v^{\varepsilon
}(s,M_{s}^{\varepsilon,t}))ds\right]  =0.
\]
Thus
\[
\tilde{\mathbb{E}}_{t}[v^{\varepsilon}(T,M_{T}^{\varepsilon,t}%
)]=v^{\varepsilon}(t,0).
\]
Similar to the proof of Theorem \ref{levy theorem}, we get%
\[
\hat{\mathbb{E}}\left[  \left \vert \hat{\mathbb{E}}_{t}[\varphi(M_{T}%
-M_{t})]-u(t,0)\right \vert \right]  \leq2C_{\varphi}\sqrt{2(\bar{\sigma}%
^{2}+1)/\pi}\sqrt{h}+2C_{\varphi}\sqrt{2(T-t)/\pi}\varepsilon,
\]
where $C_{\varphi}$ is the Lipschitz constant of $\varphi$. Letting
$h\rightarrow0$ and $\varepsilon \rightarrow0$, we obtain the desired result.
\end{proof}

\begin{remark}
It is important to note that we can easily construct a continuous symmetric
martingale $(M_{t}^{\varepsilon})_{t\geq0}$ and use It\^{o}'s formula on the
$G$-expectation space. However, we can only construct a discrete symmetric
martingale $(M_{t_{i}}^{\varepsilon})_{i=0}^{n}$ and use Taylor's expansion on
the complete consistent sublinear expectation space.
\end{remark}

\section{Reflection principle of $G$-Brownian motion}

In this section, let $\Omega=C_{0}([0,\infty))$ be the space of real-valued
continuous paths $(\omega_{t})_{t\geq0}$ with $\omega_{0}=0$. The canonical
process $(B_{t})_{t\geq0}$ is defined by%
\[
B_{t}(\omega):=\omega_{t}\text{ for }\omega \in \Omega.
\]
For each given $0\leq \underline{\sigma}^{2}\leq \bar{\sigma}^{2}$ with
$\bar{\sigma}^{2}>0$, define%
\begin{equation}
G(a):=\frac{1}{2}(\bar{\sigma}^{2}a^{+}-\underline{\sigma}^{2}a^{-})\text{ for
}a\in \mathbb{R}. \label{neweq-21}%
\end{equation}
Peng in \cite{Peng 1} constructed a sublinear expectation $\mathbb{\hat{E}%
}^{G}[\cdot]$ called $G$-expectation on $L_{ip}(\Omega)$, under which
$(B_{t})_{t\geq0}$ is a $1$-dimensional $G$-Brownian motion. Furthermore, for
any given $\tilde{G}:\mathbb{R}\rightarrow \mathbb{R}$ such that%
\begin{equation}
\left \{
\begin{array}
[c]{l}%
\tilde{G}(0)=0;\\
\tilde{G}(a)\leq \tilde{G}(b)\text{ if }a\leq b;\\
\tilde{G}(a)-\tilde{G}(b)\leq G(a-b)\text{ for }a,b\in \mathbb{R}.
\end{array}
\right.  \label{neweq-22}%
\end{equation}
By using the following PDE:%
\begin{equation}
\partial_{t}u-\tilde{G}(\partial_{xx}^{2}u)=0,\ u(0,x)=\varphi(x),
\label{neweq-23}%
\end{equation}
Peng constructed a nonlinear expectation $\mathbb{\hat{E}}^{\tilde{G}}[\cdot]$
called $\tilde{G}$-expectation on $L_{ip}(\Omega)$ satisfying the following
relation:%
\[
\mathbb{\hat{E}}^{\tilde{G}}[\varphi(B_{t})]=u^{\varphi}(t,0)\text{ for }%
t\geq0,
\]%
\begin{equation}
\mathbb{\hat{E}}^{\tilde{G}}[X]-\mathbb{\hat{E}}^{\tilde{G}}[Y]\leq
\mathbb{\hat{E}}^{G}[X-Y]\text{ for }X,Y\in L_{ip}(\Omega), \label{neweq-20}%
\end{equation}
where $u^{\varphi}$ is the solution of PDE (\ref{neweq-23}). Similar to the
definition of $\mathbb{\hat{E}}_{t}^{G}[\cdot]$ in (\ref{new-newww-2}), Peng
also define the nonlinear conditional expectation $\mathbb{\hat{E}}%
_{t}^{\tilde{G}}[\cdot]$, which still satisfies the relation
(\ref{new-newww-1}). Under the nonlinear expectation space $(\Omega
,L_{ip}(\Omega),(L_{ip}(\Omega_{t}))_{t\geq0},(\mathbb{\hat{E}}_{t}^{\tilde
{G}})_{t\geq0})$, $(B_{t})_{t\geq0}$ is a process with stationary and
independent increments, which is called a $1$-dimensional $\tilde{G}$-Brownian
motion. It is important to note that $G$ satisfies (\ref{neweq-22}), which
implies that the $\tilde{G}$-Brownian motion is a generalization of the
$G$-Brownian motion. By (\ref{neweq-20}), we can easily obtain%
\begin{equation}
\left \vert \mathbb{\hat{E}}^{\tilde{G}}[X]-\mathbb{\hat{E}}^{\tilde{G}%
}[Y]\right \vert \leq \mathbb{\hat{E}}^{G}[|X-Y|]\text{ for }X,Y\in
L_{ip}(\Omega). \label{neweq-211}%
\end{equation}
Thus $\mathbb{\hat{E}}^{\tilde{G}}[\cdot]$ can be continuously extended on
$L_{G}^{1}(\Omega)$.

Consider the following space of simple processes: for $p\geq1$,%
\[
M_{G}^{p,0}(0,T):=\left \{  \eta_{t}=\sum_{i=0}^{N-1}\xi_{i}I_{[t_{i},t_{i+1}%
)}(t):0=t_{0}<\cdots<t_{N}=T,\xi_{i}\in L_{G}^{p}(\Omega_{t_{i}})\right \}  .
\]
Denote by $M_{G}^{p}(0,T)$ (resp. $\bar{M}_{G}^{p}(0,T)$) the completion of
$M_{G}^{p,0}(0,T)$ under the norm $||\eta||_{M_{G}^{p}}:=\left(
\mathbb{\hat{E}}^{G}\left[  \int_{0}^{T}|\eta_{t}|^{p}dt\right]  \right)
^{1/p}$ (resp. $||\eta||_{\bar{M}_{G}^{p}}:=\left(  \mathbb{\hat{E}}%
^{G}\left[  \int_{0}^{T}|\eta_{t}|^{p}d\langle B\rangle_{t}\right]  \right)
^{1/p}$), where $(\langle B\rangle_{t})_{t\geq0}$ is the quadratic variation
process of $G$-Brownian motion $(B_{t})_{t\geq0}$. Peng in \cite{Peng 1}
showed that%
\[
\underline{\sigma}^{2}s\leq \langle B\rangle_{t+s}-\langle B\rangle_{t}\leq
\bar{\sigma}^{2}s\text{ for }t,s\geq0\text{.}%
\]
Thus $M_{G}^{p}(0,T)\subset \bar{M}_{G}^{p}(0,T)$ and $M_{G}^{p}(0,T)=\bar
{M}_{G}^{p}(0,T)$ under $\underline{\sigma}^{2}>0$. For each $\eta_{t}%
=\sum_{i=0}^{N-1}\xi_{i}I_{[t_{i},t_{i+1})}(t)\in M_{G}^{2,0}(0,T)$, define
the It\^{o} integral%
\[
\int_{0}^{T}\eta_{t}dB_{t}:=\sum_{i=0}^{N-1}\xi_{i}(B_{t_{i+1}}-B_{t_{i}}).
\]
Peng in \cite{peng2005, peng2008, Peng 1} obtained the following It\^{o}
equality%
\[
\mathbb{\hat{E}}^{G}\left[  \left(  \int_{0}^{T}\eta_{t}dB_{t}\right)
^{2}\right]  =\mathbb{\hat{E}}^{G}\left[  \int_{0}^{T}|\eta_{t}|^{2}d\langle
B\rangle_{t}\right]  .
\]
Thus the It\^{o} integral can be continuously extended on $\bar{M}_{G}%
^{2}(0,T)$.

The purpose of this section is to obtain the distribution of $\sup_{s\leq
t}B_{s}-B_{t}$, for this we need to use the It\^{o}-Tanaka formula. In
\cite{L-Q, Y-S-G}, the authors obtained the It\^{o}-Tanaka formula and related
properties for $G$-Brownian motion. Here we use the representation theorem for
$G$-expectation to study the It\^{o}-Tanaka formula.

The following theorem is the representation theorem for $G$-expectation.

\begin{theorem}
(see \cite{D-H-P, H-P}) There exists a weakly compact set of probability
measures $\mathcal{P}$ on $(\Omega,\mathcal{B}(\Omega))$ such that
\[
\mathbb{\hat{E}}^{G}[\xi]=\sup_{P\in \mathcal{P}}E_{P}[\xi]\text{ for all }%
\xi \in L_{G}^{1}(\Omega).
\]
$\mathcal{P}$ is called a set that represents $\mathbb{\hat{E}}^{G}$.
\end{theorem}

Let $\mathcal{P}$ be a weakly compact set that represents $\mathbb{\hat{E}%
}^{G}$. For this $\mathcal{P}$, we define capacity
\[
c(A):=\sup_{P\in \mathcal{P}}P(A),\  \ A\in \mathcal{B}(\Omega).
\]

A set $A\subset \mathcal{B}(\Omega)$ is called polar if $c(A)=0$. A property
holds "quasi-surely" (q.s.) if it holds outside a polar set. In the following,
we do not distinguish two random variables $X$ and $Y$ if $X=Y\ $q.s.

The following theorem is the well-known Krylov's estimate (see \cite{Krylov-2,
Me, Situ}).

\begin{theorem}
[Krylov's estimate]\label{Krylov2} Let $(B_{t})_{t\geq0}$ be a $1$-dimensional
$G$-Brownian motion. Then, for $T>0$, $p\geq1$ and for each Borel function
$g$,
\[
\mathbb{\hat{E}}^{G}\left[  \int_{0}^{T}|g(B_{t})|d\langle B\rangle
_{t}\right]  \leq C\left(  \int_{\mathbb{R}}|g(x)|^{p}dx\right)  ^{1/p},
\]
where $C=(\mathbb{\hat{E}}^{G}[\langle B\rangle_{T}])^{(p-1)/p}(\mathbb{\hat
{E}}^{G}[|B_{T}|])^{1/p}$.
\end{theorem}

\begin{proof}
For reader's convenience, we give a probabilistic proof. For any
$P\in \mathcal{P}$, it is easy to check that $(B_{t})_{t\geq0}$ is a martingale
under $P$. By the occupation times formula, we obtain
\begin{equation}
\int_{0}^{T}|g(B_{t})|^{p}d\langle B\rangle_{t}=\int_{\mathbb{R}}%
|g(a)|^{p}L_{T}^{P}(a)da,\text{ }P\text{-a.s.}, \label{new-nwww-1}%
\end{equation}
where $L_{T}^{P}(a)$ is the local time in $a$ of $B$ under $P$. On the other
hand, by the It\^{o}-Tanaka formula, we have%
\begin{equation}
|B_{T}-a|=|a|+\int_{0}^{T}\text{sgn}(B_{t}-a)dB_{t}+L_{T}^{P}(a),\text{
}P\text{-a.s.} \label{new-nwww-2}%
\end{equation}
Taking expectation on both sides, we get%
\begin{equation}
0\leq E_{P}[L_{T}^{P}(a)]=E_{P}[|B_{T}-a|]-|a|\leq E_{P}[|B_{T}|]\leq
\mathbb{\hat{E}}^{G}[|B_{T}|]. \label{new-nwww-3}%
\end{equation}
Combining (\ref{new-nwww-1}) and (\ref{new-nwww-3}), by H\"{o}lder's
inequality, we obtain
\begin{align*}
E_{P}\left[  \int_{0}^{T}|g(B_{t})|d\langle B\rangle_{t}\right]   &  \leq
C_{1}\left(  E_{P}\left[  \int_{0}^{T}|g(B_{t})|^{p}d\langle B\rangle
_{t}\right]  \right)  ^{1/p}\\
&  \leq C\left(  \int_{\mathbb{R}}|g(x)|^{p}dx\right)  ^{1/p},
\end{align*}
where $C_{1}=(\mathbb{\hat{E}}^{G}[\langle B\rangle_{T}])^{(p-1)/p}$ and
$C=C_{1}(\mathbb{\hat{E}}^{G}[|B_{T}|])^{1/p}$. Since $C$ is independent of
$P$, we get the desired result by taking supremum over $P\in \mathcal{P}$ in
the above inequality.
\end{proof}

Similar to Theorems 4.15 and 4.16 in \cite{H-W-Z}, we have the following
proposition which contains the case $\underline{\sigma}^{2}=0$. The proof is omitted.

\begin{proposition}
\label{Mtilde} Let $(B_{t})_{t\geq0}$ be a $1$-dimensional $G$-Brownian
motion. For each $T>0$, we have the following results:

\begin{description}
\item[(1)] If $\varphi$ is in $L^{p}\left(  \mathbb{R}\right)  $ with
$p\geq1,$ then $\left(  \varphi \left(  B_{t}\right)  \right)  _{t\leq T}%
\in \bar{M}_{G}^{1}(0,T)$. Moreover, for each $\varphi^{\prime}=\varphi$, a.e.,
we have $\left(  \varphi^{\prime}\left(  B_{t}\right)  \right)  _{t\leq
T}=\left(  \varphi \left(  B_{t}\right)  \right)  _{t\leq T}$ in $\bar{M}%
_{G}^{1}(0,T)$.

\item[(2)] Let $\left(  \varphi^{k}\right)  _{k\geq1}$ be a sequence of Borel
measurable functions such that $|\varphi^{k}(x)|\leq \bar{c}\left(
1+|x|^{l}\right)  $, $k\geq1$, for some positive constants $\bar{c}$ and $l.$
If $\varphi^{k}\rightarrow \varphi$, a.e., then for each $p\geq1,$%
\[
\lim_{k\rightarrow \infty}\mathbb{\hat{E}}^{G}\left[  \int_{0}^{T}\left \vert
\varphi^{k}(B_{t})-\varphi(B_{t})\right \vert ^{p}d\langle B\rangle_{t}\right]
=0.
\]

\item[(3)] If $\varphi$ is a Borel measurable function of polynomial growth,
then $(\varphi(B_{t}))_{t\leq T}\in \bar{M}_{G}^{2}(0,T).$
\end{description}
\end{proposition}

Now we can give the It\^{o}-Tanaka formula on the $G$-expectation space. For
each $P\in \mathcal{P}$, we have the following It\^{o}-Tanaka formula under $P$%
\begin{equation}
|B_{t}-a|=|a|+\int_{0}^{t}\text{sgn}(B_{s}-a)dB_{s}+L_{t}^{P}(a),\text{
}P\text{-a.s.} \label{new-nwww-4}%
\end{equation}
By Proposition \ref{Mtilde}, we have $($sgn$(B_{s}-a))_{s\leq t}\in \bar{M}%
_{G}^{2}(0,t)$, which implies that $\int_{0}^{t}$sgn$(B_{s}-a)dB_{s}\in
L_{G}^{2}(\Omega_{t})$ for $t\geq0$. Set%
\[
L_{t}(a)=|B_{t}-a|-|a|-\int_{0}^{t}\text{sgn}(B_{s}-a)dB_{s}\in L_{G}%
^{2}(\Omega_{t}).
\]
Then, by (\ref{new-nwww-4}), we obtain the following It\^{o}-Tanaka formula on
the $G$-expectation space
\begin{equation}
|B_{t}-a|=|a|+\int_{0}^{t}\text{sgn}(B_{s}-a)dB_{s}+L_{t}(a),\text{ q.s.,}
\label{new-nwww-5}%
\end{equation}
and $L_{t}(a)$ is called the local time in $a$ of $B$ under $\mathbb{\hat{E}%
}^{G}[\cdot]$.

\begin{lemma}
\label{le-G1}Let $(B_{t})_{t\geq0}$ be a $1$-dimensional $G$-Brownian motion.
Then $\int_{0}^{t}\text{sgn}(B_{s})dB_{s}$, $t\geq0$, is still a $G$-Brownian motion.
\end{lemma}

\begin{proof}
By Proposition \ref{Mtilde}, we have $($sgn$(B_{s}))_{s\leq t}\in \bar{M}%
_{G}^{2}(0,t)$ for each $t\geq0$. Then we obtain that $\int_{0}^{t}%
\text{sgn}(B_{s})dB_{s}\in L_{G}^{2}(\Omega_{t})$, $t\geq0$, is a symmetric
martingale, and%
\[
\langle \int_{0}^{\cdot}\text{sgn}(B_{s})dB_{s}\rangle_{t}=\int_{0}%
^{t}|\text{sgn}(B_{s})|^{2}d\langle B\rangle_{s}=\langle B\rangle_{t}.
\]
By Theorem \ref{levyG}, we get the desired result.
\end{proof}

The following theorem is the reflection principle for $G$-Brownian motion $B$.

\begin{theorem}
\label{ref-B} Let $(B_{t})_{t\geq0}$ be a $1$-dimensional $G$-Brownian motion
and $(L_{t}(0))_{t\geq0}$ be the local time of $B$ under $\mathbb{\hat{E}}%
^{G}[\cdot]$. Then%
\[
(S_{t}-B_{t},S_{t})_{t\geq0}\overset{d}{=}(|B_{t}|,L_{t}(0))_{t\geq0}%
\]
under $\mathbb{\hat{E}}^{G}[\cdot]$, where $S_{t}=\sup_{s\leq t}B_{s}$ for
$t\geq0$.
\end{theorem}

\begin{remark}
Specially, $S_{t}-B_{t}\overset{d}{=}|B_{t}|$, i.e., $\sup_{s\leq t}%
(B_{s}-B_{t})\overset{d}{=}|B_{t}|$.
\end{remark}

In order to prove this theorem, we need the following well-known Skorokhod
lemma. Let $\mathcal{D}_{0}([0,\infty))$ be the space of real-valued right
continuous with left limit (RCLL) paths $(\omega_{t})_{t\geq0}$ with
$\omega_{0}=0$.

\begin{lemma}
[Skorokhod]\label{Skorohod} Let $x\in \mathcal{D}_{0}([0,\infty))$ be given.
Then there exists a unique pair $(y,z)\in \mathcal{D}_{0}([0,\infty
);\mathbb{R}^{2})$ such that

\begin{description}
\item[(a)] $z(t)=x(t)+y(t)\geq0$ for $t\geq0$;

\item[(b)] $y$ is increasing with $y(0)=0;$

\item[(c)] $\int_{0}^{\infty}z(t)dy(t)=0.$
\end{description}

Moreover, $(y,z)$ can be expressed as
\begin{equation}
y(t)=\sup_{0\leq s\leq t}(-x(s)),\ z(t)=x(t)+\sup_{0\leq s\leq t}%
(-x(s)),\ t\geq0. \label{new-nwww-6}%
\end{equation}

\end{lemma}

\textbf{Proof of Theorem \ref{ref-B}. }For each $x\in \mathcal{D}_{0}%
([0,\infty))$, let $(y,z)\in \mathcal{D}_{0}([0,\infty);\mathbb{R}^{2})$ be
defined as in (\ref{new-nwww-6}). Define the mapping $F:$ $\mathcal{D}%
_{0}([0,\infty))\rightarrow \mathcal{D}_{0}([0,\infty);\mathbb{R}^{2})$ as%
\[
F(x)=(y,z).
\]

By the It\^{o}-Tanaka formula on the $G$-expectation space, we have
\[
|B_{t}|=\int_{0}^{t}\text{sgn}(B_{s})dB_{s}+L_{t}(0),\text{ q.s.}%
\]
Applying It\^{o}'s formula to $|B_{t}|^{2}$ and $|B_{t}|^{2}=2\int_{0}%
^{t}B_{s}dB_{s}+\langle B\rangle_{t}$, we can get%
\[
\int_{0}^{\infty}|B_{s}|dL_{s}(0)=0,\text{ q.s.}%
\]
Thus, by Lemma \ref{Skorohod}, we obtain%
\begin{equation}
F[(\int_{0}^{t}\text{sgn}(B_{s})dB_{s})_{t\geq0}]=(L_{t}(0),|B_{t}|)_{t\geq
0},\text{ q.s.} \label{new-nwww-7}%
\end{equation}
On the other hand,%
\[
S_{t}-B_{t}=-B_{t}+S_{t}.
\]
Noting that $S_{t}=\sup_{0\leq s\leq t}(-(-B_{s}))$, by Lemma \ref{Skorohod},
we have%
\begin{equation}
F[(-B_{t})_{t\geq0}]=(S_{t},S_{t}-B_{t})_{t\geq0}. \label{new-nwww-8}%
\end{equation}

By Lemma \ref{le-G1}, we know
\begin{equation}
(\int_{0}^{t}\text{sgn}(B_{s})dB_{s})_{t\geq0}\overset{d}{=}(B_{t})_{t\geq
0}\overset{d}{=}(-B_{t})_{t\geq0}. \label{new-nwww-9}%
\end{equation}
By (\ref{new-nwww-6}), it is easy to check that $F$ is Lipschitz continuous
under the uniform topology. Thus, by (\ref{new-nwww-7}), (\ref{new-nwww-8})
and (\ref{new-nwww-9}), we can deduce the desired result. $\Box$

Let $\tilde{G}:\mathbb{R}\rightarrow \mathbb{R}$ satisfy (\ref{neweq-22}). In
the following, we extend the reflection principle to $\tilde{G}$-Brownian
motion $B$. We need the following lemma.

\begin{lemma}
\label{le-G2}Let $(B_{t})_{t\geq0}$ be a $1$-dimensional $\tilde{G}$-Brownian
motion. Then $\int_{0}^{t}\text{sgn}(B_{s})dB_{s}$, $t\geq0$, is still a
$\tilde{G}$-Brownian motion.
\end{lemma}

Under nonlinear expectation space, we do not obtain the L\'{e}vy's martingale
characterization theorem. The key problem lies in the fact that the solution
of the following PDE%
\[
\partial_{t}u-\tilde{G}(\partial_{xx}^{2}u)-\frac{1}{2}\varepsilon^{2}%
\partial_{xx}^{2}u=0,\ u(0,x)=\varphi(x)
\]
may not be regular. In the following, we give a new proof. First we need the
following lemma.

\begin{lemma}
\label{sgnB} Suppose $\underline{\sigma}^{2}>0$ in $G$. Let $\pi_{n}%
=\{t_{0}^{n},t_{1}^{n},\ldots,t_{n}^{n}\}$ be a partition of $[0,T]$ and
$\mu(\pi_{n})$ be the diameter of $\pi_{n}$. Then%
\[
\mathbb{\hat{E}}^{G}\left[  \displaystyle \int_{0}^{T}\left \vert
\sum \limits_{i=0}^{n-1}\text{sgn}(B_{t_{i}^{n}})I_{[t_{i}^{n},t_{i+1}^{n}%
)}(t)-\text{sgn}(B_{t})\right \vert ^{2}dt\right]  \rightarrow0\text{ as }%
\mu(\pi_{n})\rightarrow0.
\]

\end{lemma}

\begin{proof}
For each fixed $\varepsilon>0$, define%
\begin{equation}
\phi_{\varepsilon}(x)=I_{[\varepsilon,\infty)}(x)+\frac{x}{\varepsilon
}I_{(-\varepsilon,\varepsilon)}(x)-I_{(-\infty,-\varepsilon]}(x)\text{ for
}x\in \mathbb{R}. \label{new-nvcc-1}%
\end{equation}
It is easy to verify that $|\phi_{\varepsilon}(x)-\phi_{\varepsilon}%
(x^{\prime})|\leq \frac{1}{\varepsilon}|x-x^{\prime}|$ and $|\text{sgn}%
(x)-\phi_{\varepsilon}(x)|\leq I_{(-\varepsilon,\varepsilon)}(x)$. By Example
3.9 in {\cite{H-W-Z}}, we have
\[
\hat{\mathbb{E}}[I_{(-\varepsilon,\varepsilon)}(B_{t})]\leq \exp(\frac
{1}{2\overline{\sigma}^{2}})\frac{\varepsilon^{2\alpha}}{t^{\alpha}},
\]
where $\alpha=\frac{\underline{\sigma}^{2}}{2\overline{\sigma}^{2}}\in
(0,\frac{1}{2})$. Then
\begin{align*}
&  \mathbb{\hat{E}}^{G}\left[  \displaystyle \int_{0}^{T}\left \vert
\sum \limits_{i=0}^{n-1}\text{sgn}(B_{t_{i}^{n}})I_{[t_{i}^{n},t_{i+1}^{n}%
)}(t)-\text{sgn}(B_{t})\right \vert ^{2}dt\right] \\
&  \leq3\left \{  \mathbb{\hat{E}}^{G}\left[  \int_{0}^{T}\left \vert
I_{1}^{\varepsilon,n}(t)\right \vert ^{2}dt\right]  +\mathbb{\hat{E}}%
^{G}\left[  \int_{0}^{T}\left \vert I_{2}^{\varepsilon,n}(t)\right \vert
^{2}dt\right]  +\mathbb{\hat{E}}^{G}\left[  \int_{0}^{T}\left \vert
I_{3}^{\varepsilon}(t)\right \vert ^{2}dt\right]  \right \} \\
&  \leq3\left \{  \sum_{i=0}^{n-1}\mathbb{\hat{E}}^{G}\left[  I_{(-\varepsilon
,\varepsilon)}(B_{t_{i}^{n}})\right]  (t_{i+1}^{n}-t_{i}^{n})+\sum_{i=0}%
^{n-1}\frac{1}{\varepsilon^{2}}\int_{t_{i}^{n}}^{t_{i+1}^{n}}\mathbb{\hat{E}%
}^{G}\left[  |B_{t}-B_{t_{i}^{n}}|^{2}\right]  dt+\int_{0}^{T}\mathbb{\hat{E}%
}^{G}\left[  I_{(-\varepsilon,\varepsilon)}(B_{t})\right]  dt\right \} \\
&  \leq3\left \{  \exp(\frac{1}{2\overline{\sigma}^{2}})\varepsilon^{2\alpha
}\sum_{i=0}^{n-1}\frac{1}{(t_{i}^{n})^{\alpha}}(t_{i+1}^{n}-t_{i}^{n}%
)+\frac{\overline{\sigma}^{2}}{\varepsilon^{2}}\mu(\pi_{n})T+\exp(\frac
{1}{2\overline{\sigma}^{2}})\varepsilon^{2\alpha}\int_{0}^{T}\frac
{1}{t^{\alpha}}dt\right \}  ,
\end{align*}
where%
\[%
\begin{array}
[c]{l}%
I_{1}^{\varepsilon,n}(t)=\sum_{i=0}^{n-1}\text{sgn}(B_{t_{i}^{n}}%
)I_{[t_{i}^{n},t_{i+1}^{n})}(t)-\sum_{i=0}^{n-1}\phi_{\varepsilon}%
(B_{t_{i}^{n}})I_{[t_{i}^{n},t_{i+1}^{n})}(t),\\
I_{2}^{\varepsilon,n}(t)=\sum_{i=0}^{n-1}\phi_{\varepsilon}(B_{t_{i}^{n}%
})I_{[t_{i}^{n},t_{i+1}^{n})}(t)-\phi_{\varepsilon}(B_{t}),\\
I_{3}^{\varepsilon}(t)=\phi_{\varepsilon}(B_{t})-\text{sgn}(B_{t}).
\end{array}
\]
Letting $\mu(\pi_{n})\rightarrow0$ in the above inequality, we get%
\begin{align*}
&  \limsup_{\mu(\pi_{n})\rightarrow0}\mathbb{\hat{E}}^{G}\left[
\displaystyle \int_{0}^{T}\left \vert \sum \limits_{i=0}^{n-1}\text{sgn}%
(B_{t_{i}^{n}})I_{[t_{i}^{n},t_{i+1}^{n})}(t)-\text{sgn}(B_{t})\right \vert
^{2}dt\right] \\
&  \leq3\exp(\frac{1}{2\overline{\sigma}^{2}})\varepsilon^{2\alpha}\left \{
\lim_{\mu(\pi_{n})\rightarrow0}\sum_{i=0}^{n-1}\frac{1}{(t_{i}^{n})^{\alpha}%
}(t_{i+1}^{n}-t_{i}^{n})+\int_{0}^{T}\frac{1}{t^{\alpha}}dt\right \} \\
&  =6\exp(\frac{1}{2\overline{\sigma}^{2}})\varepsilon^{2\alpha}\int_{0}%
^{T}\frac{1}{t^{\alpha}}dt.
\end{align*}
Taking $\varepsilon \rightarrow0$, we obtain the desired result.
\end{proof}

\textbf{Proof of Lemma \ref{le-G2}. }By the proof of Lemma \ref{le-G1}, we
know that $\int_{t}^{T}$sgn$(B_{s})dB_{s}\in L_{G}^{2}(\Omega_{T})$ for $0\leq
t<T<\infty$. By the property of $\mathbb{\hat{E}}_{t}^{\tilde{G}}[\cdot]$, we
only need to prove that%
\begin{equation}
\mathbb{\hat{E}}_{t}^{\tilde{G}}[\varphi(\int_{t}^{T}\text{sgn}(B_{s}%
)dB_{s})]=u^{\varphi}(T-t,0), \label{new-newww-3}%
\end{equation}
where $0<t<T$, $\varphi \in C_{b.Lip}(\mathbb{R})$ and $u^{\varphi}$ is the
solution of PDE (\ref{neweq-23}). The proof is divided into two steps.

\textbf{Step 1:} We first prove (\ref{new-newww-3}) under the case
$\underline{\sigma}^{2}>0$. By Lemma \ref{sgnB}, we have%
\begin{equation}%
\begin{array}
[c]{l}%
\mathbb{\hat{E}}^{G}\left[  \left \vert \mathbb{\hat{E}}_{t}^{\tilde{G}%
}[\displaystyle \varphi(\sum_{i=0}^{n-1}\text{sgn}(B_{t_{i}^{n}}%
)(B_{t_{i+1}^{n}}-B_{t_{i}^{n}}))]-\mathbb{\hat{E}}_{t}^{\tilde{G}}%
[\varphi(\int_{t}^{T}\text{sgn}(B_{s})dB_{s})]\right \vert ^{2}\right] \\
\leq C_{\varphi}^{2}\mathbb{\hat{E}}^{G}\left[  \displaystyle \left \vert
\sum_{i=0}^{n-1}\text{sgn}(B_{t_{i}^{n}})(B_{t_{i+1}^{n}}-B_{t_{i}^{n}}%
)-\int_{t}^{T}\text{sgn}(B_{s})dB_{s}\right \vert ^{2}\right] \\
\leq C_{\varphi}^{2}\bar{\sigma}^{2}\mathbb{\hat{E}}^{G}\left[
\displaystyle \int_{t}^{T}\left \vert \sum \limits_{i=0}^{n-1}\text{sgn}%
(B_{t_{i}^{n}})I_{[t_{i}^{n},t_{i+1}^{n})}(s)-\text{sgn}(B_{s})\right \vert
^{2}ds\right] \\
\rightarrow0,\text{ as }n\rightarrow \infty,
\end{array}
\label{new-newww-4}%
\end{equation}
where $C_{\varphi}$ is the Lipschitz constant of $\varphi$. On the other hand,
noting that $-(B_{t_{i+1}^{n}}-B_{t_{i}^{n}})\overset{d}{=}(B_{t_{i+1}^{n}%
}-B_{t_{i}^{n}})$, by the property of $\mathbb{\hat{E}}_{t}^{\tilde{G}}%
[\cdot]$, we get%
\[
\mathbb{\hat{E}}_{t_{i}^{n}}^{\tilde{G}}[\phi(\text{sgn}(B_{t_{i}^{n}%
})(B_{t_{i+1}^{n}}-B_{t_{i}^{n}}))]=\mathbb{\hat{E}}_{t_{i}^{n}}^{\tilde{G}%
}[\phi(B_{t_{i+1}^{n}}-B_{t_{i}^{n}})]\text{ for any }\phi \in C_{b.Lip}%
(\mathbb{R}).
\]
From this, we can easily obtain%
\begin{equation}
\mathbb{\hat{E}}_{t}^{\tilde{G}}[\varphi(\sum_{i=0}^{n-1}\text{sgn}%
(B_{t_{i}^{n}})(B_{t_{i+1}^{n}}-B_{t_{i}^{n}}))]=\mathbb{\hat{E}}_{t}%
^{\tilde{G}}[\varphi(B_{T}-B_{t})]=u^{\varphi}(T-t,0). \label{new-newww-5}%
\end{equation}
Combining (\ref{new-newww-4}) and (\ref{new-newww-5}), we deduce
(\ref{new-newww-3}).

\textbf{Step 2:} Set $\tilde{\Omega}=C_{0}([0,\infty);\mathbb{R}^{2})$, the
corresponding canonical process is denoted by $(B_{t},\bar{B}_{t})_{t\geq0}$.
Define%
\[
G^{\prime}\left(  \left[
\begin{array}
[c]{cc}%
a & b\\
b & c
\end{array}
\right]  \right)  =G(a)+\frac{1}{2}c,\text{ }\tilde{G}^{\prime}\left(  \left[
\begin{array}
[c]{cc}%
a & b\\
b & c
\end{array}
\right]  \right)  =\tilde{G}(a)+\frac{1}{2}c\text{ for }a,b,c\in \mathbb{R}.
\]
Following Peng \cite{Peng 1}, we can construct $G^{\prime}$-expectation
$(\mathbb{\hat{E}}_{t}^{G^{\prime}})_{t\geq0}$ and $\tilde{G}^{\prime}%
$-expectation $(\mathbb{\hat{E}}_{t}^{\tilde{G}^{\prime}})_{t\geq0}$ via the
following PDE:
\[
\partial_{t}u-G(\partial_{xx}^{2}u)-\frac{1}{2}\partial_{yy}^{2}%
u=0,\ u(0,x,y)=\psi(x,y),
\]%
\[
\partial_{t}u-\tilde{G}(\partial_{xx}^{2}u)-\frac{1}{2}\partial_{yy}%
^{2}u=0,\ u(0,x,y)=\psi(x,y).
\]
It is easy to verify that $\mathbb{\hat{E}}_{t}^{G^{\prime}}[\xi
]=\mathbb{\hat{E}}_{t}^{G}[\xi]$, $\mathbb{\hat{E}}_{t}^{\tilde{G}^{\prime}%
}[\xi]=\mathbb{\hat{E}}_{t}^{\tilde{G}}[\xi]$ for each $\xi \in L_{G}%
^{1}(\Omega)$ and $(B_{t}+\delta \bar{B}_{t})_{t\geq0}$ is a $G_{\delta}%
$-Brownian motion under $\mathbb{\hat{E}}_{t}^{G^{\prime}}[\cdot]$ for each
$\delta>0$, where $G_{\delta}(a)=G(a)+\frac{1}{2}\delta^{2}a$ for
$a\in \mathbb{R}$. By Step 1, we have%
\begin{equation}%
\begin{array}
[c]{rl}%
\mathbb{\hat{E}}_{t}^{\tilde{G}^{\prime}}[\varphi(\int_{t}^{T}\text{sgn}%
(B_{s}+\delta \bar{B}_{s})d(B_{s}+\delta \bar{B}_{s}))] & =\mathbb{\hat{E}}%
_{t}^{\tilde{G}^{\prime}}[\varphi(B_{T}+\delta \bar{B}_{T}-B_{t}-\delta \bar
{B}_{t})]\\
& =\mathbb{\hat{E}}^{\tilde{G}^{\prime}}[\varphi(B_{T}+\delta \bar{B}_{T}%
-B_{t}-\delta \bar{B}_{t})].
\end{array}
\label{new-nvcc-2}%
\end{equation}
For each fixed $\varepsilon>0$, define $\phi_{\varepsilon}$ as in
(\ref{new-nvcc-1}). By Theorem \ref{Krylov2}, we have
\begin{align*}
&  \mathbb{\hat{E}}^{G^{\prime}}\left[  \left \vert \mathbb{\hat{E}}%
_{t}^{\tilde{G}^{\prime}}[\varphi(\int_{t}^{T}\text{sgn}(B_{s}+\delta \bar
{B}_{s})d(B_{s}+\delta \bar{B}_{s}))]-\mathbb{\hat{E}}_{t}^{\tilde{G}}%
[\varphi(\int_{t}^{T}\text{sgn}(B_{s})dB_{s})]\right \vert ^{2}\right] \\
&  =\mathbb{\hat{E}}^{G^{\prime}}\left[  \left \vert \mathbb{\hat{E}}%
_{t}^{\tilde{G}^{\prime}}[\varphi(\int_{t}^{T}\text{sgn}(B_{s}+\delta \bar
{B}_{s})d(B_{s}+\delta \bar{B}_{s}))]-\mathbb{\hat{E}}_{t}^{\tilde{G}^{\prime}%
}[\varphi(\int_{t}^{T}\text{sgn}(B_{s})dB_{s})]\right \vert ^{2}\right] \\
&  \leq C_{\varphi}^{2}\mathbb{\hat{E}}^{G^{\prime}}\left[  \left \vert
\int_{t}^{T}\text{sgn}(B_{s}+\delta \bar{B}_{s})d(B_{s}+\delta \bar{B}_{s}%
)-\int_{t}^{T}\text{sgn}(B_{s})dB_{s}\right \vert ^{2}\right] \\
&  \leq4C_{\varphi}^{2}\mathbb{\hat{E}}^{G^{\prime}}\left[  \left \vert
\int_{t}^{T}I_{1}^{\varepsilon,\delta}(s)d(B_{s}+\delta \bar{B}_{s})\right \vert
^{2}+\left \vert \int_{t}^{T}I_{2}^{\varepsilon,\delta}(s)dB_{s}\right \vert
^{2}+\left \vert \delta \int_{t}^{T}\phi_{\varepsilon}(B_{s}+\delta \bar{B}%
_{s})d\bar{B}_{s}\right \vert ^{2}+\left \vert \int_{t}^{T}I_{3}^{\varepsilon
}(s)dB_{s}\right \vert ^{2}\right] \\
&  \leq4C_{\varphi}^{2}\left \{  \mathbb{\hat{E}}^{G^{\prime}}\left[  \int
_{t}^{T}I_{(-\varepsilon,\varepsilon)}(B_{s}+\delta \bar{B}_{s})d\langle
B+\delta \bar{B}\rangle_{s}\right]  +\frac{\delta^{2}\bar{\sigma}^{2}T^{2}%
}{2\varepsilon^{2}}+\delta^{2}T+\mathbb{\hat{E}}^{G^{\prime}}\left[  \int
_{t}^{T}I_{(-\varepsilon,\varepsilon)}(B_{s})d\langle B\rangle_{s}\right]
\right \} \\
&  \leq4C_{\varphi}^{2}\left \{  2\mathbb{\hat{E}}^{G^{\prime}}[|B_{T}+\bar
{B}_{T}|]\varepsilon+\frac{\bar{\sigma}^{2}T^{2}+2\varepsilon^{2}%
T}{2\varepsilon^{2}}\delta^{2}+2\mathbb{\hat{E}}^{G}[|B_{T}|]\varepsilon
\right \}  ,
\end{align*}
where%
\[%
\begin{array}
[c]{l}%
I_{1}^{\varepsilon,\delta}(s)=\text{sgn}(B_{s}+\delta \bar{B}_{s}%
)-\phi_{\varepsilon}(B_{s}+\delta \bar{B}_{s}),\\
I_{2}^{\varepsilon,\delta}(s)=\phi_{\varepsilon}(B_{s}+\delta \bar{B}_{s}%
)-\phi_{\varepsilon}(B_{s}),\\
I_{3}^{\varepsilon}(s)=\phi_{\varepsilon}(B_{s})-\text{sgn}(B_{s}).
\end{array}
\]
Taking $\delta \downarrow0$ in the above inequality, we obtain%
\begin{align*}
&  \limsup_{\delta \downarrow0}\mathbb{\hat{E}}^{G^{\prime}}\left[  \left \vert
\mathbb{\hat{E}}_{t}^{\tilde{G}^{\prime}}[\varphi(\int_{t}^{T}\text{sgn}%
(B_{s}+\delta \bar{B}_{s})d(B_{s}+\delta \bar{B}_{s}))]-\mathbb{\hat{E}}%
_{t}^{\tilde{G}}[\varphi(\int_{t}^{T}\text{sgn}(B_{s})dB_{s})]\right \vert
^{2}\right] \\
&  \leq4C_{\varphi}^{2}\left \{  2\mathbb{\hat{E}}^{G^{\prime}}[|B_{T}+\bar
{B}_{T}|]\varepsilon+2\mathbb{\hat{E}}^{G}[|B_{T}|]\varepsilon \right \}  ,
\end{align*}
which implies%
\begin{equation}
\lim_{\delta \downarrow0}\mathbb{\hat{E}}^{G^{\prime}}\left[  \left \vert
\mathbb{\hat{E}}_{t}^{\tilde{G}^{\prime}}[\varphi(\int_{t}^{T}\text{sgn}%
(B_{s}+\delta \bar{B}_{s})d(B_{s}+\delta \bar{B}_{s}))]-\mathbb{\hat{E}}%
_{t}^{\tilde{G}}[\varphi(\int_{t}^{T}\text{sgn}(B_{s})dB_{s})]\right \vert
^{2}\right]  =0 \label{new-nvcc-3}%
\end{equation}
by letting $\varepsilon \downarrow0$. It is easy to check that%
\begin{equation}
\lim_{\delta \downarrow0}\left \vert \mathbb{\hat{E}}^{\tilde{G}^{\prime}%
}[\varphi(B_{T}+\delta \bar{B}_{T}-B_{t}-\delta \bar{B}_{t})]-\mathbb{\hat{E}%
}^{\tilde{G}}[\varphi(B_{T}-B_{t})]\right \vert =0. \label{new-nvcc-4}%
\end{equation}
By (\ref{new-nvcc-2}), (\ref{new-nvcc-3}) and (\ref{new-nvcc-4}), we obtain
(\ref{new-newww-3}). $\Box$

Similar to the proof of Theorem \ref{ref-B}, we can immediately obtain the
following reflection principle for $\tilde{G}$-Brownian motion $B$. The proof
is omitted.

\begin{theorem}
\label{ref-B1} Let $(B_{t})_{t\geq0}$ be a $1$-dimensional $\tilde{G}%
$-Brownian motion and $(L_{t}(0))_{t\geq0}$ be the local time of $B$ under
$\mathbb{\hat{E}}^{G}[\cdot]$. Then%
\[
(S_{t}-B_{t},S_{t})_{t\geq0}\overset{d}{=}(|B_{t}|,L_{t}(0))_{t\geq0},
\]
under $\mathbb{\hat{E}}^{\tilde{G}}[\cdot]$, where $S_{t}=\sup_{s\leq t}B_{s}$
for $t\geq0$.
\end{theorem}

\begin{remark}
In particular, $(\sup_{s\leq t}(B_{s}-B_{t}))_{t\geq0}\overset{d}{=}%
(|B_{t}|)_{t\geq0}$ under $\mathbb{\hat{E}}^{\tilde{G}}[\cdot]$.
\end{remark}

\end{document}